\documentclass[11pt, final]{article}
\usepackage[top=0.7in, bottom=0.7in, left=1.2in, right=1.2in]{geometry}

\include{macro_list}

\RequirePackage{amsthm}
\theoremstyle{plain} 
\newtheorem{theorem}{Theorem}[section]

\newtheorem{lemma}      [theorem]{Lemma}
\newtheorem{proposition}[theorem]{Proposition}

\usepackage{cprotect}

\newdimen\iwidth
\newdimen\iheight

\usepackage{bigstrut}
\usepackage{multirow}

\usepackage{amsmath,amssymb} 
\usepackage{algorithm,algorithmic}

\usepackage{enumitem}
\usepackage{color}

\newcommand{\jmc}[1]{{\color{blue} #1}}

\newtheorem{prop}{Proposition}
\numberwithin{algorithm}{section}

\RequirePackage[capitalize,nameinlink]{cleveref}[0.19]

\crefname{section}{section}{sections}
\crefname{subsection}{subsection}{subsections}

\Crefname{figure}{Figure}{Figures}

\Crefname{prop}{Proposition}{Propositions}

\crefformat{equation}{\textup{#2(#1)#3}}
\crefrangeformat{equation}{\textup{#3(#1)#4--#5(#2)#6}}
\crefmultiformat{equation}{\textup{#2(#1)#3}}{ and \textup{#2(#1)#3}}
{, \textup{#2(#1)#3}}{, and \textup{#2(#1)#3}}
\crefrangemultiformat{equation}{\textup{#3(#1)#4--#5(#2)#6}}%
{ and \textup{#3(#1)#4--#5(#2)#6}}{, \textup{#3(#1)#4--#5(#2)#6}}{, and \textup{#3(#1)#4--#5(#2)#6}}

\Crefformat{equation}{#2Equation~\textup{(#1)}#3}
\Crefrangeformat{equation}{Equations~\textup{#3(#1)#4--#5(#2)#6}}
\Crefmultiformat{equation}{Equations~\textup{#2(#1)#3}}{ and \textup{#2(#1)#3}}
{, \textup{#2(#1)#3}}{, and \textup{#2(#1)#3}}
\Crefrangemultiformat{equation}{Equations~\textup{#3(#1)#4--#5(#2)#6}}%
{ and \textup{#3(#1)#4--#5(#2)#6}}{, \textup{#3(#1)#4--#5(#2)#6}}{, and \textup{#3(#1)#4--#5(#2)#6}}

\crefdefaultlabelformat{#2\textup{#1}#3}

\title{Uncertainty quantification in large Bayesian linear inverse problems using Krylov subspace methods}
 \author{Arvind K. Saibaba\thanks{Department of Mathematics, North Carolina State University, Raleigh, NC, asaibab@ncsu.edu}
    \and
  Julianne Chung\thanks{Department of Mathematics, Computational Modeling and Data Analytics Division, Academy of Integrated Science, Virginia Tech, Blacksburg, VA, jmchung@vt.edu.}
 \and
  Katrina Petroske \thanks{Department of Mathematics, North Carolina State University, Raleigh, NC, kepetros@ncsu.edu.}
}

\begin{document}
\maketitle
\begin{abstract}
  {Uncertainty quantification for linear inverse problems remains a challenging task, especially for problems with a very large number of unknown parameters (e.g., dynamic inverse problems) and for problems where computation of the square root and inverse of the prior covariance matrix are not possible (e.g., those from the Mat\'{e}rn class).  In this work, we assume that generalized Golub-Kahan based methods have been used to compute an estimate of the solution, and we describe efficient methods to explore the posterior distribution. By exploiting the generalized Golub-Kahan bidiagonalization, we get an approximation of the posterior covariance matrix for ``free.'' We provide theoretical results that quantify the accuracy of the {approximate posterior covariance matrix} and of the resulting posterior distribution. Then, we describe efficient methods that use the approximation to compute measures of uncertainty, including the Kullback-Liebler divergence. We present two methods that use preconditioned Lanczos methods to efficiently generate samples from the posterior distribution. Numerical examples from dynamic photoacoustic tomography demonstrate the effectiveness of the described approaches.}
\end{abstract}



\section{Introduction}
Inverse problems arise in various scientific applications, and a significant amount of effort has focused on developing efficient and robust methods to compute approximate solutions.  However, as these numerical solutions are increasingly being used for data analysis and to aid in decision-making, there is a critical need to be able to obtain valuable uncertainty information (e.g., solution variances, samples, and credible intervals) to assess the reliability of computed solutions.  Tools for inverse uncertainty quantification (UQ) often build upon the Bayesian framework from statistical inverse problems.  Great overviews and introductions can be found in, e.g., \cite{biegler2011large,tenorio17,tenorio2011data,kaipio2006statistical,calvetti2007introduction}.

Unfortunately, for very large inverse problems, UQ using the Bayesian approach is prohibitively expensive from a computational standpoint. This is partly because the posterior covariance matrices are so large that constructing, storing, and working with them directly are not computationally feasible. For these scenarios, a hybrid generalized Golub-Kahan based method was proposed in~\cite{chung2017generalized} to compute Tikhonov regularized solutions efficiently and to select a regularization parameter simultaneously {and automatically}. In this work, we go beyond computing reconstructions (e.g., maximum a posteriori (MAP) estimates) and develop efficient methods for inverse {UQ}. We focus on methods {that use the approximate posterior distribution to compute measures of uncertainty} and develop preconditioned iterative solvers to efficiently sample from the posterior distribution by exploiting various tools from numerical linear algebra.

For concreteness, we consider linear inverse problems of the form
\begin{equation}
 \label{eq:linear problem}
 \bfd = \bfA\bfs + \bfdelta,
\end{equation}
where the goal is to reconstruct the desired parameters $\bfs \in \bbR^n$, given matrix $\bfA \in \bbR^{m\times n}$ and the observed data $\bfd \in \bbR^m$. Typically, $\bfA$ is an ill-conditioned matrix that models the forward process, and we assume that it is known exactly.  We adopt a Bayesian approach where we assume that the measurement errors $\bfdelta$ and the unknowns $\bfs$ are  mutually independent Gaussian variables, i.e., $\bfdelta \sim \calN(\bfzero, \bfR)$ and $\bfs \sim \calN(\bfmu, \lambda^{-2}\bfQ)$
where $\bfR$ and $\bfQ$ are {symmetric} positive definite matrices, $\bfmu \in \bbR^n$, and $\lambda$ is a scaling parameter also known as the regularization parameter.  For the problems of interest, computing the inverse and square root of $\bfR$ are inexpensive, but explicit computation of $\bfQ$ (or its inverse or square root) may not be possible.  However, we assume that matrix-vector multiplications (matvecs) involving $\bfA$, $\bfA\t$, and $\bfQ$ can be done efficiently.

Recall Bayes' theorem of inverse problems, which states that the posterior probability distribution function is given by
\begin{equation*} \pi_{\rm post}  = \pi(\bfs|\bfd) = \frac{\pi(\bfd|\bfs)\pi(\bfs)}{\pi(\bfd)}\,.
\end{equation*}
Under our assumptions, the posterior distribution has the following representation,
\begin{equation}
 \label{eq:post}
 \pi_{\rm post}\propto \exp\left(-\frac{1}{2}\|\bfA \bfs - \bfd \|_{\bfR^{-1}}^2 -  \frac{\lambda^2}{2}\| \bfs-\bfmu\|_{\bfQ^{-1}}^2\right),
\end{equation}
where $\norm{\bfx}{\bfM} = \sqrt{\bfx\t \bfM \bfx}$ is a vector norm for any symmetric positive definite matrix $\bfM$. Thus, the posterior distribution is Gaussian{, with corresponding measure $\rho_\text{post} = \calN(\spost,\post)$},
where the posterior covariance and mean are given as
\begin{equation}\label{eqn:postcovmean}
\post \equiv (\lambda^2\bfQ^{-1} + \bfA\t\bfR^{-1}\bfA)^{-1} \quad\text{and}\quad \bfs_{\rm post} =  \post(\bfA\t\bfR^{-1} \bfd + \lambda^2\bfQ^{-1} \bfmu)
\end{equation}
respectively \cite{calvetti2007introduction}. In the Bayesian framework, the solution to the inverse problem is the posterior distribution.
However, for practical interpretation and data analysis, it is necessary to describe various characteristics of the posterior distribution \cite{tenorio17}.

We now describe what sets our work apart from previous work on inverse UQ.  Typical approaches model the inverse of the prior covariance matrix (known as the precision matrix) as a discretized partial differential operator (e.g., Laplacian). This results in a sparse precision matrix that is relatively easy to factorize or solve linear systems with. In contrast, we model the prior covariance matrix entry-wise using covariance kernels (e.g., $\gamma$-exponential, or Mat\'ern class), which allows the user the flexibility to incorporate a wide range of {prior models (e.g., nonisotropic or spatiotemporal).} {The main} challenge is that the resulting prior covariance matrices are dense; explicitly forming and factorizing these matrices is prohibitively expensive. For such prior models, efficient matrix-free techniques (e.g., FFT and $\mc{H}$-matrix approaches) can be used to compute matvecs with the prior covariance matrix $\bfQ$. {However, new algorithms need to be developed to perform inverse UQ in these cases, and we address that in this paper.} Specifically, we develop Krylov subspace algorithms {that exploit the generalized Golub-Kahan bidiagonalization} for approximating the posterior covariance matrix and for sampling from the posterior distribution.

\paragraph{Overview of main contributions}
The main point of this paper is to compute uncertainty measures involving the posterior distribution by storing bases for the Krylov subspaces during the computation of the MAP estimate and reusing the information contained in these subspaces for inverse UQ. The main contributions are as follows:
\begin{itemize}[leftmargin=*]
 \item We propose an approximation to the posterior covariance matrix using the generalized Golub-Kahan approach that has an efficient representation (low-rank perturbation of the prior covariance matrix). We develop error bounds for monitoring the accuracy of the approximate posterior covariance matrix, based on the generalized Golub-Kahan iterates.
\item  We relate the error in the approximate posterior covariance matrix to the error in the approximate posterior distribution. We also show how to efficiently compute measures of uncertainty, such as the Kullback-Leibler divergence between the posterior and the prior distributions.
 \item We develop {two different algorithms} for generating samples from the posterior distribution using preconditioned Lanczos methods. {The first algorithm uses the approximate posterior covariance matrix, whereas the second algorithm uses the true covariance matrix but in different ways.}
\end{itemize}

The idea of using low-rank perturbative approximations for the posterior covariance matrix previously appeared in~\cite{flath2011fast,bui2012extreme, bui2013computational, spantini2015optimal}; however, these approaches rely on the ability to work with the square root or an appropriate factorization of $\bfQ$ (or its inverse). The authors in~\cite{bui2012extreme,saibaba2015fastc} use randomized approaches to efficiently compute a low-rank approximation; in particular, the algorithm in \cite{saibaba2015fastc} does not require a factorization of $\bfQ$. However, theoretical bounds suggest that randomized algorithms are effective when the singular values decay sufficiently rapidly. This assumption is valid for moderately or severely ill-posed inverse problems; however, for tomography based applications, which we consider in this paper, the {decay of the} singular values is not sufficiently rapid, and therefore we pursue Krylov subspace methods.
Previous work on Lanczos methods for sampling from Gaussian distributions {can be found in, e.g.,}~\cite{parker2012sampling,schneider2003krylov,simpson2008krylov,chow2014preconditioned}, but these algorithms are meant for sampling from generic Gaussian distributions and do not exploit the structure of the posterior covariance matrix as we do.

The paper is organized as follows.  In \cref{sec:background}, we provide a brief overview of the generalized Golub-Kahan bidiagonalization and preconditioning \jmc{of Krylov methods for sampling}.  Then, in \cref{sec:postapprox}, we use elements from the generalized Golub-Kahan bidiagonalization to approximate the posterior covariance matrix and provide theoretical bounds for the approximation.  Not only are these bounds of interest for subsequent analysis and sampling, but they can also be used to determine a good stopping criterion for the iterative methods.  In \cref{sec:sampling} we describe efficient Krylov subspace samplers for sampling from the posterior distribution.  Numerical results for large inverse problems from image \jmc{reconstruction} are provided in \cref{sec:numerics}, and conclusions and future work are provided in \cref{sec:conclusions}.


\section{Background}
\label{sec:background}
In this section, we provide a brief background on two core topics that will be heavily used in the development of efficient methods to explore the posterior.  In \cref{ss_genGK}, we review an iterative hybrid method based on the generalized Golub-Kahan bidiagonlization that can be used to approximate the MAP estimate, which amounts to minimizing the negative log likelihood of the posterior probability distribution function, i.e.
\begin{equation}\bfs_{\rm post}  = \argmin_{\bfs \in \mathbb{R}^n}\> -\log \pi_\text{post}
	 = \argmin_{\bfs \in \mathbb{R}^n} \> \frac{1}{2}\|\bfA \bfs -\bfd \|_{\bfR^{-1}}^2 +  \frac{\lambda^2}{2}\| \bfs-\bfmu\|_{\bfQ^{-1}}^2\,.
\label{eqn:LS_orig}
\end{equation}
Notice that with a change of variables, $\bfs_{\rm post} = \bfmu + \bfQ \bfx$ where $\bfx$ is the solution to
\begin{equation}
 \label{eqn:LS_tik_x}
 \min_{\bfx \in \mathbb{R}^n} \frac{1}{2}\|\bfA \bfQ \bfx - \bfb\|_{\bfR^{-1}}^2 + \frac{\lambda^2}{2} \|\bfx\|_{\bfQ}^2
\end{equation}
where $\bfb = \bfd - \bfA \bfmu.$ This change of variables is motivated by the fact that factorizing {and/or inverting $\bfQ$ is infeasible in many applications. For more details on choices of }prior covariance matrices $\bfQ$ for which this holds, we refer the reader to the discussion in our previous works \cite[Section 2.1]{chung2017generalized} and \cite[Section 2.3]{chung2018efficient}.
For readers familiar with hybrid Krylov iterative methods, \cref{ss_genGK} can be skipped.  Then in \cref{ss_krylov}, we review preconditioned Krylov subspace solvers for generating samples from normal distributions.

\subsection{Generalized hybrid iterative methods}
\label{ss_genGK}
Here, we provide an overview of {the hybrid method based on the generalized} Golub-Kahan (gen-GK) bidiagonalization, but refer the interested reader to~\cite{chung2017generalized,arioli2013generalized} for more details.

The basic idea behind the generalized hybrid methods is first to generate a basis $\bfV_k$ for the Krylov subspace
\begin{equation} \mc{S}_k \equiv \text{Span}\{\bfV_k\} = \krylov{k}(\bfA\t\bfR^{-1}\bfA\bfQ,\bfA\t\bfR^{-1} \bfb),
\end{equation}
where $\krylov{k}(\bfM,\bfg)=\text{Span}\{\bfg,\bfM\bfg,\dots,\bfM^{k-1}\bfg\}$, and second to solve~\eqref{eqn:LS_tik_x} in this subspace. A basis for $\mc{S}_k$ can be generated using the gen-GK bidiagonalization process\footnote{Generalized Golub-Kahan methods were first proposed by Benbow~\cite{benbow1999solving} for generalized least squares problems, and used in several applications, see e.g.~\cite{arioli2013iterative,arioli2013generalized,orban2017iterative}. However, the specific form of the bidiagonalization was developed in~\cite{chung2017generalized}.} summarized in \cref{alg:wlsqr}, where
at the end of $k$ steps, we have the matrices
\begin{equation}
\label{eq:GKelements}
	\bfU_{k+1} \equiv [\bfu_1,\dots,\bfu_{k+1}], \bfV_k \equiv [\bfv_1,\dots,\bfv_k], \quad \mbox{and} \quad
\bfB_k \equiv \> \begin{bmatrix}
\alpha_1 \\ \beta_2 & \ddots \\ & \ddots & \alpha_k \\ & & \beta_{k+1}
\end{bmatrix}
\end{equation}
that in exact arithmetic satisfy
\begin{equation}
	\label{eq:genGKrelations}
\bfA \bfQ \bfV_k = \bfU_{k+1}\bfB_k,  \quad \bfA\t \bfR^{-1} \bfU_{k+1} =  \> \bfV_k \bfB_k\t + \alpha_{k+1}\bfv_{k+1}\bfe_{k+1}\t
\end{equation}
and
\begin{equation}
 \label{eq:orthogonality}
\bfU_{k+1}\t \bfR^{-1} \bfU_{k+1} = \bfI_{k+1},  \quad  \bfV_k\t \bfQ \bfV_k = \bfI_k.
\end{equation}
Vector $\bfe_{k+1}$ corresponds to the $(k+1)$st standard unit vector.

\begin{algorithm}[!ht]
\begin{algorithmic}[1]
	\ENSURE{[$\bfU_k$, $\bfV_k$, $\bfB_k$] = \texttt{gen-GK}($\bfA$, $\bfR$, $\bfQ$, $\bfb$, $k$) }
\STATE $\beta_1 \bfu_1 = \bfb,$ where $\beta_1 = \norm{\bfb}{\bfR^{-1}}$
\STATE $\alpha_1 \bfv_1 = \bfA\t \bfR^{-1}\bfu_1,$ where $\alpha_1 =  \norm{\bfA\t \bfR^{-1} \bfu_{1}}{\bfQ}$
\FOR {$i=1, \dots, k$}
\STATE $\beta_{i+1}\bfu_{i+1} = \bfA\bfQ\bfv_i - \alpha_i \bfu_i$, where $\beta_{i+1} = \norm{\bfA\bfQ\bfv_i - \alpha_i \bfu_i}{\bfR^{-1}}$
\STATE $\alpha_{i+1}\bfv_{i+1} = \bfA\t \bfR^{-1} \bfu_{i+1} - \beta_{i+1} \bfv_i$, where $\alpha_{i+1} = \norm{\bfA\t \bfR^{-1} \bfu_{i+1} - \beta_{i+1} \bfv_i}{\bfQ}$
\ENDFOR
\end{algorithmic}
\caption{gen-GK bidiagonlization}
\label{alg:wlsqr}
\end{algorithm}

We seek {an approximate solution to} \cref{eqn:LS_tik_x} of the form $\bfx_{k}  = \bfV_k \bfz_{k} $, so that $\bfx_k \in \mc{S}_k$, where the coefficients $\bfz_k$ can be determined by solving the following problem,
\begin{equation}
	\label{e_wlsqr}
	\min_{\bfx_k \in \mc{S}_k } \>\frac{1}{2}\norm{\bfA\bfQ\bfx_k-\bfb}{\bfR^{-1}}^2+\frac{\lambda^2}{2}\norm{\bfx_{k}}{\bfQ}^2\quad \Leftrightarrow \quad \min_{\bfz_k \in \mb{R}^k} \> \frac{1}{2}\normtwo{ \bfB_k \bfz_k-\beta_1 \bfe_1 }^2 + \frac{\lambda^2}{2} \normtwo{\bfz_k}^2,
\end{equation}
where the equivalency uses the relations in {\cref{eq:genGKrelations} and \cref{eq:orthogonality}}.
For fixed $\lambda$, an approximate MAP estimate can be recovered by undoing the change of variables,
\begin{equation}\label{eqn:undo_change}
\bfs_{k} =  \bfmu + \bfQ \bfx_{k}  = \bfmu + \bfQ\bfV_k (\bfB_k\t \bfB_k + \lambda^2 \bfI)^{-1} \bfB_k \t \beta_1 \bfe_1\,,
\end{equation}
where now $\bfs_{k} \in \bfmu + \bfQ\mc{S}_k$. If $\lambda$ is not known \emph{a priori}, a hybrid approach can be used where sophisticated SVD based methods are applied to the right equation in \cref{e_wlsqr}.  In this work, we use the hybrid implementation described in \cite{chung2017generalized} called genHyBR. {The benefit of using this hybrid approach is that this algorithm automatically determines the number of iterations $k$ and the regularization parameter $\lambda.$}

\subsection{Sampling from a Gaussian distribution}\label{ss_krylov}
Let
{$\bar\bfnu \in \bbR^n$}
and let $\bfGamma \in \bbR^{n\times n}$ be any {symmetric positive} definite matrix. Suppose the goal is to obtain samples from the Gaussian distribution $\calN({\bar\bfnu}, \bfGamma)$. Throughout this paper, let $\bfepsilon\sim \calN(\bf0,\bfI)$.  If we have or are able to obtain a factorization of the form $\bfGamma = \bfS \bfS\t,$ then
\[ {\bfnu = \bar\bfnu} + \bfS \bfepsilon \]
is a sample from $\mc{N}({\bar\bfnu},\bfGamma)$, since $\mathbb{E}[\bfnu] =  {\bar\bfnu}$  and
\[ \text{Cov}(\bfnu) = \mathbb{E}[({\bfnu-\bar\bfnu)(\bfnu-\bar\bfnu})\t] = \mathbb{E}[\bfS\bfepsilon \bfepsilon\bfS\t] = \bfGamma. \]
Note that any matrix $\bfS$ that satisfies $\bfS\bfS\t = \bfGamma$ can be used to generate samples. We show how Krylov subspace solvers, in particular preconditioned versions, can be used to efficiently generate approximate samples from $\mc{N}(\bf0,\bfGamma)$ and $\mc{N}(\bf0,\bfGamma^{-1})$. These approaches will be extended for sampling from the posterior in \cref{sec:sampling}.

Given $\bfGamma$ and starting guess $\bfepsilon$, after $k$ steps of the symmetric Lanczos process, we have matrix $\bfW_{k}= [\bfw_1,\dots,\bfw_{k}]\in \bbR^{n \times k}$ that contains orthonormal columns and tridiagonal matrix
\[ \bfT_k = \bmat{\gamma_1 & \delta_2  \\ \delta_2 & \gamma_2 & \delta_2 \\  &  \ddots & \ddots & \ddots \\ & &  \delta_{k-1}& \gamma_{k-1} & \delta_k \\ & & & \delta_k& \gamma_k} \in \bbR^{k \times k}   \]
such that in exact arithmetic we have the following relation,
\[ \bfGamma \bfW_k = \bfW_{k}{\bfT}_k + \delta_{k+1}\bfw_{k+1}\bfe_k\t\,.\]
The Lanczos process is summarized in \cref{alg:Lanczossampling}. Computed matrices $\bfW_k$ and $\bfT_k$ can then be used to obtain approximate draws from $\mc{N}(\bf0,\bfGamma)$ and $\mc{N}(\bf0,\bfGamma^{-1})$ as
\begin{equation}
 \bfxi_k=  \>  \bfW_k\bfT_k^{1/2} \delta_1 \bfe_1 \quad \mbox{and} \quad \bfzeta_k =  \>  \bfW_k\bfT_k^{-1/2} \delta_1 \bfe_1,
\end{equation}
respectively.

\begin{algorithm}[!ht]
\begin{algorithmic}[1]
	\ENSURE{[$\bfW_k$, $\bfT_k$] = \texttt{Lanczos}($\bfGamma$, $\bfepsilon$, $k$)}
\STATE $\delta_0 = 1, \bfw_0 = \bfzero, \delta_1 = \normtwo{\bfepsilon}, \bfw_1 = \bfepsilon/\delta_1$
\FOR {$i=1, \dots, k$}
\STATE $\gamma_i = \bfw_i\t \bfGamma \bfw_i $,
\STATE $\bfr = \bfGamma \bfw_i- \gamma_i \bfw_i- \delta_{i-1} \bfw_{i-1}$
\STATE $\bfw_{i+1} = \bfr/\delta_i$, where $\delta_i = \| \bfr\|_2$
\ENDFOR
\end{algorithmic}
\caption{Lanczos tridiagonalization}
\label{alg:Lanczossampling}
\end{algorithm}

\paragraph{Convergence.} The approximation improves as $k$ increases, and we expect typical convergence behavior for the Lanczos process whereby convergence to extremal (i.e., largest and smallest) eigenvalues will be fast. The following result~\cite[Theorem 3.3]{simpson2008krylov} sheds light onto the convergence of Krylov subspace methods for sampling. The error in the sample $\bfzeta_k$ is given by
\[ \|\bfGamma^{-1/2}\bfepsilon - \bfzeta_k\|_2 \leq \sqrt{\lambda_\text{min}(\bfGamma)} \|\bfr_k\|_2,\]
where $\lambda_\text{min}(\bfGamma)$ is the smallest eigenvalue of $\bfGamma$. The term $\bfr_k= \bf\epsilon - \bfA\bfx_k$ is the residual vector at the $k$-th iteration of the conjugate gradient method and $\bfx_k = \bfV_k\bfT_k^{-1}\delta_1 \bfe_1$. The residual vector $\|\bfr_k\|_2$ can be bounded using standard techniques in Krylov subspace methods~\cite{saad2003iterative}. To use this as a stopping criterion, we note that $\|\bfr_k\|_2 = \delta_1 |\bfe_k\t\bfT_k^{-1}\bfe_1|$
 and by the Cauchy interlacing theorem $\lambda_\text{min}(\bfGamma) \leq \lambda_\text{min}(\bfT_k)$. Combining the two bounds we have
\[ \|\bfGamma^{-1/2}\bfepsilon - \bfzeta_k\|_2 \leq \sqrt{\lambda_\text{min}(\bfT_k)} \delta_1 |\bfe_k\t\bfT_k^{-1}\bfe_1|.\]
However, in numerical experiments we found that the bound was too pessimistic and instead adopted the approach in~\cite{chow2014preconditioned}. Suppose we define the relative error norm as
\[ e_k = \frac{\|\bfzeta_k - \bfGamma^{-1/2}\bfepsilon\|_2}{\|\bfGamma^{-1/2}\bfepsilon\|_2}.\]
In practice, this quantity cannot be computed, but it can be estimated using successive iterates as
\[ \tilde{e}_k = \frac{\|\bfzeta_k - \bfzeta_{k+1}\|_2}{\|\bfzeta_{k+1}\|_2}.\]
When convergence is fast, we found this bound to be more representative of the true error in numerical experiments. The downside is that computing this is expensive since it costs $\mathcal{O}(nk^2)$ flops. However, this cost can be avoided by first writing
\[ \bfzeta_k = \bfW_k \widehat\bfzeta_k{, \quad \mbox{where} \quad} \widehat\bfzeta_k = \delta_1 \bfT_k^{-1/2}\bfe_1.\]
Since the columns of $\bfW_k$ are orthonormal, then
\begin{equation}
	\label{eq:ektilde}
 \tilde{e}_k = \frac{\|\widehat\bfzeta_k'- \widehat\bfzeta_{k+1}\|_2}{\|\widehat\bfzeta_{k+1}\|_2} \qquad \bfzeta_k'\equiv \begin{bmatrix} \widehat\bfzeta_k \\ 0 \end{bmatrix}.
 \end{equation}
Therefore, $\tilde{e}_k$ can be computed in $\mathcal{O}(k^3)$ operations rather than $\mathcal{O}(nk^2)$ operations. A similar approach can be used to monitor the convergence of $\bfxi_k$ to $\bfGamma^{1/2}\bfepsilon$.

\paragraph{Preconditioning.} It is well known that {an appropriate preconditioner can significantly accelerate} convergence of Krylov subspace methods for solving linear systems. Assume that we have a preconditioner $\bfG$ which satisfies
{$\bfGamma^{-1} \approx \bfG\t\bfG$. }
Then, the same preconditioner can be used to accelerate the convergence of Krylov subspace methods for generating samples, as we now show. {Let
\[ \bfS = \bfG^{-1}(\bfG\bfGamma\bfG^\top)^{1/2} \quad \mbox{ and } \quad \bfT = \bfG^\top(\bfG\bfGamma\bfG^\top)^{-1/2}\,, \]
then it is easy to see that
\[ \bfS \bfS\t = \bfG^{-1}(\bfG\bfGamma\bfG^\top)^{1/2} (\bfG\bfGamma\bfG^\top)^{1/2} \bfG^{-\top}  = \bfGamma\]
and similarly $\bfT \bfT\t = \bfGamma^{-1}$.}
 The Lanczos process is then applied to $\bfG\bfGamma\bfG\t$ and approximate samples from $\mc{N}(\bf0,\bfGamma)$ and $\mc{N}(\bf0,\bfGamma^{-1})$ can be obtained by computing
\begin{equation}
 \bfxi_k =  \>  \bfG^{-1}\bfW_k\bfT_k^{1/2} \delta_1 \bfe_1 \qquad \bfzeta_k =  \>  \bfG\t\bfW_k\bfT_k^{-1/2} \delta_1 \bfe_1.
\end{equation}
If $\bfG$ is a good preconditioner, in the sense that $\bfGamma^{-1} \approx \bfG\t\bfG$ (alternatively, $\bfG\bfGamma\bfG^\top \approx \bfI$), then the Krylov subspace method is expected to converge rapidly. The choice of preconditioner depends on the specific problem; we comment on the choice of preconditioners in the numerical experiments in \cref{sec:numerics}.


\section{Approximating the posterior distribution using the gen-GK {bidiagonalization}}
\label{sec:postapprox}
The basic goal of this work is to enable exploration of the posterior distribution for large-scale inverse problems by exploiting
elements and relationships from the {gen-GK} bidiagonalization (c.f., equations {\cref{eq:GKelements}--\cref{eq:orthogonality}}) to approximate the posterior covariance matrix $\post$.

Consider computing an approximate eigenvalue decomposition of $\bfH = \bfA\t\bfR^{-1}\bfA$. We define the Ritz pairs $(\theta,\bfy)$  obtained as the solution of the following eigenvalue problem,
\[ (\bfH\bfQ\bfV_k\bfy - \theta \bfV_k\bfy) \> \perp_{\bfQ}\> \Span{\bfV_k} . \]
Here the orthogonality condition $\perp_{\bfQ}$ is defined {with respect to} the weighted inner product $\langle\cdot,\cdot\rangle_{\bfQ}$. From {\cref{eq:genGKrelations,eq:orthogonality}}, the Ritz pairs can be obtained by the solution of the eigenvalue problem
\[ \bfB_k\t\bfB_k \bfy_j = \theta_j \bfy_j \qquad j = 1,\dots,k.\]
The Ritz pairs can be combined to express the eigenvalue decomposition {in matrix form as},
\[ \bfB_k\t \bfB_k = \bfY_k\bfTheta_k\bfY_k\t. \]
The accuracy of the Ritz pairs can be quantified by the residual norm, defined as
\[  \norm{\bfr_j}{\bfQ} \equiv \> \norm{\bfH\bfQ\bfV_k\bfy_j - \theta_j \bfV_k\bfy_j}{\bfQ} = \alpha_{k+1}{\beta}_{k+1}  |\bfe_k\t\bfy_j| \qquad j=1,\dots,k. \]
Furthermore, using arguments from~\cite[Theorem 11.4.2]{parlett1980symmetric} it can be shown that
\[ \bfT_k \equiv \bfB_k\t\bfB_k =  \min_{\bfDelta \in \mb{R}^{k\times k} } \norm{\bfH\bfQ\bfV_k - \bfV_k\bfDelta }{\bfQ} \]
is the best approximation over the subspace $\mc{S}_k \equiv \krylov{k}(\bfH\bfQ,\bfA\t\bfR^{-1}\bfb)$. {Thus,}
the best low-rank approximation of $\bfH$ over the space $\mc{S}_k$ {is given by $\bfH \approx  \bfV_k\bfT_k \bfV_k\t$}.
Here we define the matrix $\|\cdot\|_{\bfQ}$ norm to be $\norm{\bfM}{\bfQ} = \max_{\|\bfx\|_2=1} \|\bfM\bfx\|_{\bfQ}$.

An approximation of this kind has been previously explored in~\cite{saibaba2015fastc,flath2011fast,bui2012extreme,bui2013computational}; however, the error estimates developed in the above references assume that the exact eigenpairs are available. If the Ritz pairs converge to the exact eigenpairs, then furthermore, the optimality result in~\cite[Theorem 2.3]{spantini2015optimal} applies here as well.

For the rest of this paper, we use the following low-rank approximation of $\bfH$ which is constructed using the {gen-GK} bidiagonalization
\begin{equation}\label{eqn:lowrank}
\widehat\bfH \equiv \bfV_k\bfT_k\bfV_k\t.
\end{equation}
Using this low-rank approximation, we can define the \emph{approximate posterior distribution}
{$\widehat\pi_\text{post}${, with the corresponding measure $\widehat\rho_\text{post}= \mathcal{N}(\bfs_k, \posth)$,} which is a Gaussian distribution with  covariance matrix
\begin{equation}
\label{eqn:posth}
 \posth \equiv (\lambda^2 \bfQ^{-1} + {\widehat \bfH})^{-1} \end{equation}
and mean $\bfs_k$ defined in \cref{eqn:undo_change}. Using \cref{eqn:posth}, we note that
\begin{equation}
  \label{eqn:meanhat}
  \bfs_k = \bfmu + \posth \bfA\t \bfR^{-1}\bfb.
\end{equation}
See \cref{sec:appendix4} for the derivation.
}
\subsection{Posterior covariance approximation}\label{ssec:postacc} First, we derive a way to monitor the accuracy of the low-rank approximation using the information available from the {gen-GK} bidiagonalization. This result is similar to~\cite[Proposition 3.3]{simon2000low}.
\begin{prop}\label{prop:recur} Let $\bfH_{\bfQ} = \bfQ^{1/2}\bfH\bfQ^{1/2}$ and $\widehat\bfH_{\bfQ} = \bfQ^{1/2}\widehat\bfH\bfQ^{1/2}$. After $k$ steps of \cref{alg:wlsqr}, the error in the low-rank approximation $\widehat\bfH$, measured as
\begin{equation}
\omega_{k} = \|\bfH_{\bfQ} - \widehat{\bfH}_{\bfQ}\|_{F},
\end{equation}
satisfies the recurrence
\[\omega_{k+1}^2 = \omega_k^2 - 2|\alpha_{k+1}\beta_{k+1}|^2 - |\alpha_{k+1}^2 + \beta_{k+2}^2|^2.\]
\end{prop}
\begin{proof} See \cref{sec:appendix1}. \end{proof}
This proposition shows that, in exact arithmetic, the error in the low-rank approximation $\widehat\bfH$ to $\bfH$ decreases monotonically as the iterations progress. Estimates for $\omega_k$ can be obtained in terms of the singular values of $\bfR^{-1/2}\bfA\bfQ^{1/2}$ following the approach in~\cite[Theorem 3.2]{simon2000low} and~\cite[Theorem 2.7]{huang2017some}. However, we do not pursue them here.

Given the low-rank approximation, we can define the approximate posterior covariance $\posth$ in \cref{eqn:posth}.
The recurrence relation in \cref{prop:recur} can be used to derive the following error estimates for $\post$.
\begin{theorem}\label{thm:post} The approximate posterior covariance matrix $\posth$ satisfies
\[ \|\post - \posth\|_F \leq   \lambda^{-2} \min\left\{\omega_k \lambda^{-2} \|\bfQ\|_2,  \frac{\omega_k \|\bfQ\|_F} {\lambda^2 + \omega_k}\right\} .\]
\end{theorem}
\begin{proof} See \cref{sec:appendix1}. \end{proof}

The above theorem quantifies the error in the posterior covariance matrix in the Frobenius norm. However, the authors in~\cite{spantini2015optimal} argue that the Frobenius norm is not the appropriate metric to measure the distance between covariance matrices. Instead, they advocate the F\"orstner distance since it respects the geometry of the cone of positive definite covariance metrics. We take a different approach and consider metrics between the approximate and the true posterior distributions.

\subsection{Accuracy of posterior distribution}\label{ssec:acc}
The Kullback-Leibler (KL) divergence is a measure of ``distance'' between two different probability measures. The KL divergence is not a true metric on the set of probability measures, since it is not symmetric and does not satisfy the triangle inequality~\cite{sullivan2015introduction}. Despite these short-comings, the KL divergence is widely used since it has many favorable properties. Both the true and the approximate posterior measures are Gaussian, so the KL divergence between these measures takes the form (using~\cite[Exercise 5.2]{sullivan2015introduction}):
\[ {D_\text{KL}(\widehat\rho_\text{post}\| \rho_\text{post})} = \frac12\left[\trace(\post^{-1}\posth) + \| \spost-{\bfs_k}\|_{\post^{-1}}^2 -n + \log\frac{\det\post}{\det\posth} \right]. \]
We first present a result that can be used to monitor the accuracy of the trace of $\bfH_{\bfQ}$.
 \begin{proposition}\label{prop:trecur}
 Let $\theta_k = \trace(\bfH_{\bfQ} - \widehat\bfH_{\bfQ})$. Then $\theta_k$ satisfies the recurrence relation
 \[  \theta_{k+1} =\theta_k   - (\alpha_{k+1}^2 + \beta_{k+2}^2). \]
 \end{proposition}
 \begin{proof}
 See \cref{sec:appendix2}.
 \end{proof}
Note that the Cauchy interlacing theorem implies that $\theta_k $ is non-negative; therefore, as with \cref{prop:recur}, this result implies that $\theta_k$ is monotonically decreasing.

\begin{theorem}\label{thm:dklacc} At the end of $k$ iterations, the KL divergence between the true and the approximate posterior distributions satisfies
\[  0 \leq {D_\text{KL}(\widehat\rho_\text{post}\| \rho_\text{post})} \leq \frac{\lambda^{-2}}{2}\left[ \theta_k + \frac{\omega_k^2}{\lambda^2 + \omega_k}\alpha_1^2\beta_1^2\right].\]
\end{theorem}
 \begin{proof}
 See \cref{sec:appendix2}.
 \end{proof}
Both $\theta_k$ and $\omega_k$ are monotonically decreasing, implying that the accuracy of the estimator for the KL divergence improves as the iterations progress. This theorem can be useful in providing bounds for the error using other metrics.

For example, consider the Hellinger metric and Total Variation (TV) distance denoted by  {$d_\mc{H} (\rho_\text{post}, \widehat\rho_\text{post})$} and {$d_\text{TV} (\rho_\text{post}, \widehat\rho_\text{post})$}  respectively. Combining Pinsker's inequality~\cite[Theorem 5.4]{sullivan2015introduction} and Kraft's inequality~\cite[Theorem 5.10]{sullivan2015introduction}, we have the following relationship
 \begin{equation} \label{eqn:hellinger} {d_\mc{H}^2 (\rho_\text{post}, \widehat\rho_\text{post}) \leq d_\text{TV} (\rho_\text{post}, \widehat\rho_\text{post}) \leq \sqrt{2D_\text{KL}(\widehat\rho_\text{post}\| \rho_\text{post})} } . \end{equation}
 {Thus,} \cref{thm:dklacc} can be used to find upper bounds for the Hellinger metric and the TV distance between the true and approximate posterior distributions.
 Furthermore, suppose $f:(\mathbb{R}^n, \|\cdot \|_{\mathbb{R}^n}) \rightarrow (\mathbb{R}^d, \|\cdot \|_{\mathbb{R}^d})$ is a function with finite second moments {with respect to} both measures, then {by}~\cite[Proposition 5.12]{sullivan2015introduction}
 \[ {\|\mathbb{E}_{\rho_\text{post}}\,[f] -\mathbb{E}_{\widehat\rho_\text{post}}\,[f]\|_{\mathbb{R}^n} \leq 2 \sqrt{\mathbb{E}_{\rho_\text{post}}\,[\|f\|^2_{\mathbb{R}^d}]  +\mathbb{E}_{\widehat\rho_\text{post}}\,[\|f\|^2_{\mathbb{R}^d}]} d_\mc{H}({\rho_\text{post}},{\widehat\rho_\text{post}}).  }\]
 This implies that the error in the expectation of a function computed using the approximate posterior instead of the true posterior can be bounded by combining~\eqref{eqn:hellinger} and \cref{thm:dklacc}.

\subsection{Computation of information-theoretic metrics}\label{ssec:info}
 In addition to providing a measure of distance between the true and approximate posterior distributions, the KL divergence can also be used to measure the information gain between the prior and the posterior distributions. Similar to the derivation in \cref{ssec:acc} since both {$\rho_\text{prior} = \calN(\bfmu,\lambda^{-2}\bfQ)$ and $\rho_\text{post}$} are Gaussian,
 the KL divergence takes the form
\[\begin{aligned}
{D_\text{KL}(\rho_\text{post} \| \rho_\text{prior} )}  = & \> \frac12\bigg[ \trace (\lambda^2\bfQ^{-1}\post) + \lambda^{2}(\spost-\bfmu)\t\bfQ^{-1}(\spost-\bfmu) \\
& \qquad  -n - \log\det(\lambda^2\bfQ^{-1}\post) \bigg]\\
= & \frac12\bigg[ \trace (\bfI + \lambda^{-2}\bfH_{\bfQ})^{-1} +  \lambda^{2}(\spost-\bfmu)\t\bfQ^{-1}(\spost-\bfmu)\\
& \qquad  -n + \log\det( \bfI + \lambda^{-2}\bfH_{\bfQ})\bigg] \,.
\end{aligned}
\]
{Then, using the approximations generated by the gen-GK bidiagonlization, we consider the approximation
$$D_\text{KL} \equiv {D_\text{KL}(\rho_\text{post} \| \rho_\text{prior} ) \approx D_\text{KL}(\widehat\rho_\text{post} \| \rho_\text{prior} ) }\equiv \widehat{D}_\text{KL} .$$}
Using the fact that $\log\det( \bfI + \lambda^{-2}\widehat\bfH_{\bfQ}) = \log\det(\bfI + \lambda^{-2}\bfT_k),$
$$ \trace (\bfI + \lambda^{-2}\widehat \bfH_{\bfQ})^{-1}  = n - \trace ( \bfT_k(\bfT_k + \lambda^2\bfI )^{-1}),$$
and
$$ \|{\bfs_k}-\bfmu\|_{\bfQ^{-1}} = \|\bfQ\bfV_k (\bfB_k\t \bfB_k + \lambda^2 \bfI)^{-1} \bfB_k \t \beta_1 \bfe_1 \|_{\bfQ^{-1}} = \|\alpha_1\beta_1 (\bfT_k + \lambda^2\bfI)^{-1}\bfe_1 \|_2^2\,,$$
we get
\[ \begin{aligned}
\widehat{D}_\text{KL} = & \frac12\left[- \trace ( \bfT_k(\bfT_k + \lambda^2\bfI )^{-1}) + \lambda^2 \|\bfz_k\|_2^2  + \log\det(\bfI + \lambda^{-2}\bfT_k) \right]
\end{aligned} \]
where $\bfz_k = \alpha_1\beta_1 (\bfT_k + \lambda^2\bfI)^{-1}\bfe_1$.
Note that all {of} the terms only involve $k\times k$ tridiagonal matrices and, therefore,
{$\widehat{D}_\text{KL}$}
can be computed in $\mc{O}(k^3)$ once the gen-GK bidiagonalization has been computed.

The following result quantifies the accuracy of the estimator for the KL divergence between the posterior and the prior. Notice that the bound is similar to \cref{thm:dklacc}.
\begin{theorem}\label{thm:dkl}
The error in the KL divergence, in exact arithmetic, is given by
\[ |D_\text{KL} - \widehat{D}_\text{KL}| \leq \lambda^{-2}\left[\theta_k +   \frac{\lambda^2\omega_k}{\lambda^2+\omega_k} \alpha_1^2 \beta_1^2 \right],\]
where $\omega_k$ and $\theta_k$ were defined in \cref{prop:recur,prop:trecur} respectively.
\end{theorem}
\begin{proof}
See \cref{sec:appendix2}.
\end{proof}

Related to the KL divergence is the D-optimal criterion for optimal experimental design, which is defined as
\[ \phi_D \equiv \log\det(\post) - \log\det(\lambda^{-2}\bfQ) = \log\det(\bfI + \lambda^{-2}\bfH_\bfQ).\]
The D-optimal criterion can be seen as the expected KL divergence, with the expectation taken over the posterior distribution. A precise statement of this result was derived in~\cite[Theorem 1]{alexanderian2017efficient}. Similar to the KL divergence, we can estimate the D-optimal criterion as
\[ \widehat\phi_D = \log\det(\bfI + \lambda^{-2}\bfT_k).\]
From the proof of \cref{thm:dkl}, it can be readily seen that a bound for the error in the D-optimal criterion is {given by}
\[ | \phi_D- \widehat\phi_D| \leq \lambda^{-2} \theta_k. \]


\section{Sampling from the posterior distribution}
\label{sec:sampling}
Since the posterior distribution is very high-dimensional, visualizing this distribution is challenging. A popular method is to generate samples from the posterior distribution (also sometimes known as conditional realizations), which provides a family of solutions and can be used for quantifying the reconstruction uncertainty. For instance, to compute the expected value of a  quantity of interest ${q}(\cdot)$, defined as
\[ {{\mathcal{Q}}  \equiv \mb{E}\,[{q}(\bfs) \mid \bfd]= \int_{\bbR^n} {q}(\bfs) \pi(\bfs \mid \bfd) d\bfs}.\]
 Suppose, we have samples {$\{\bfs^{(j)}\}_{j=1}^N$} then ${\mathcal{Q}}_{N} \equiv N^{-1}\sum_{j=1}^{N}{q}(\bfs^{(j)})$ is the  {\em Monte Carlo estimate} of ${\mathcal{Q}}$.
 Furthermore, the Monte Carlo estimate converges to the expected value of the quantity of interest, i.e.,  ${\mathcal{Q} }_{N} \rightarrow {\mathcal{Q} }$ as $N\rightarrow\infty$ {almost surely, by the strong law of large numbers}.

We now show how to draw samples from the posterior distribution $\mc{N}(\spost,\post)$.
As described in \cref{ss_krylov}, if $\bfepsilon \sim \mc{N}(\bf0,\bfI)$ and $\bfS\bfS\t = \post$, then  \[ \bfs = \spost + \bfS\bfepsilon \]
is a sample from $\mc{N}(\spost,\post)$.
However, computing the posterior covariance matrix $\post$ and its factorization $\bfS$ is infeasible for reasons described before.
{We use preconditioned Krylov subspace methods to generate samples from the posterior distribution.}
A direct application of the approach in \cref{ss_krylov} to the posterior covariance matrix is expensive since it involves application of  $\bfQ^{-1}$. To avoid this, we present several reformulations. {The first approach we describe computes a low-rank approximation of $\bfH$ using the gen-GK approach and then uses this low-rank approximation to generate samples from the \emph{approximate} posterior distribution. Any low-rank approximation can be used, provided it is sufficiently accurate. On the other hand, the second approach generates approximate samples from the \emph{exact} posterior distribution. Both methods use a preconditioner, albeit in different ways.

Before describing our proposed methods, we briefly review a few methods for sampling from high-dimensional Gaussian distributions. The idea of using Krylov subspace methods for sampling from Gaussian random processes seems to have originated from~\cite{schneider2003krylov}. Variants of this idea have also been proposed in~\cite{parker2012sampling,chow2014preconditioned} and have found applications in Bayesian inverse problems in~\cite{gilavert2015efficient,simpson2008krylov}. The use of a low-rank surrogate of $\bfH_\bfQ$ has also been explored in~\cite{bui2012extreme,bui2013computational} and is similar to  Method 1 (c.f., \cref{ssec:method1}) that we propose. Other approaches to sampling from the posterior distribution include {randomize-then-optimize} (RTO)~\cite{bardsley2014randomize,bardsley2015randomize} and {randomized MAP} approach~\cite{wang2018randomized}. However, none of these methods can handle the case where $\bfQ^{-1}$ or $\bfQ^{-1/2}$ are not available.

\subsection{Method 1: {Sampling from $\widehat\pi_\text{post}$}}\label{ssec:method1}
{Consider generating samples from $\pi_\text{post}$, where $\post = ({\lambda^2}\bfQ^{-1} + \bfH)^{-1}$ is the posterior covariance matrix.}
Given a preconditioner $\bfG$, which we assume to be invertible, we can write
\[ \bfQ^{-1} = \bfG\t(\bfG\bfQ\bfG\t)^{-1}\bfG. \]
{Then consider the factorization} ${\lambda^2}\bfQ^{-1} = \bfL\t\bfL$ where
\begin{equation}\label{e_L}
\bfL \equiv {\lambda \,} (\bfG\bfQ\bfG\t)^{-1/2}\bfG .
\end{equation}
An important point to note is that, while writing such an factorization, we do not propose to compute it explicitly.  Instead, we access} it in a matrix-free fashion using techniques from {\cref{alg_sampling1}}.

  Plugging this into the expression for the posterior covariance, we obtain
\[ \post = (\bfL\t\bfL + \bfH)^{-1}  = \bfL^{-1}( \bfI +  \bfL^{-\top}\bfH\bfL^{-1})^{-1}\bfL^{-\top}.\]
The low-rank approximation of $\bfH$ in \cref{eqn:lowrank} can be used to derive an approximate factorization of the posterior covariance matrix
\begin{equation}
  \label{eq:Shatdefinition}
  \posth = \widehat\bfS\widehat\bfS\t\mbox{ \quad  where \quad } \widehat\bfS \equiv  \bfL^{-1}( \bfI +  \bfL^{-\top}\widehat\bfH\bfL^{-1})^{-1/2} .
\end{equation}
To efficiently compute matvecs with $\widehat\bfS$, we first compute the low-rank representation
\[ \bfL^{-\top}\widehat\bfH\bfL^{-1} = \bfZ_k\bfTheta_k\bfZ_k\t.\]
Computing the low-rank representation is accomplished using \cref{alg:lowrank}.

\begin{algorithm}[!ht]
\begin{algorithmic}[1]
\ENSURE{[$\bfZ,\bfTheta$] = \texttt{Lowrank}($\bfW$) for {an arbitrary} $\bfW \in \mathbb{R}^{n\times k}$ with $k \leq n$}
\STATE Compute thin-QR factorization $\bfQ\bfR = \bfW$
\STATE Compute eigenvalue decomposition $\bfR\bfR\t = \bfY\bfTheta\bfY\t$
\STATE Compute $\bfZ = \bfQ\bfY$
\end{algorithmic}
\caption{Low-rank representation $\bfZ\bfTheta\bfZ\t=\bfW\bfW\t$}
\label{alg:lowrank}
\end{algorithm}

 Computing matvecs with $\bfL$ (including its inverse and transpose) is done using the preconditioned Lanczos method described in \cref{ss_krylov}.  We can compute the square root of the inverse of $\bfI +\bfZ_k\bfTheta_k\bfZ_k\t$ using a variation of the {Woodbury identity~\cite[Equation (0.7.4.1)]{horn2012matrix}}
\begin{equation*} (\bfI +\bfZ_k\bfTheta_k\bfZ_k\t)^{-1/2} = \bfI -\bfZ_k\bfD_k\bfZ_k\t \qquad  \bfD_k = \bfI_k \pm (\bfI_k+\bfTheta_k)^{-1/2}.\end{equation*}
  {In summary, the procedure for computing samples $\bfxi^{(j)} \sim \calN(\bfzero, \posth)$ is provided}
  in \cref{alg_sampling1}. The accuracy of the generated samples is discussed in \cref{ssec:disc}.
\begin{algorithm}[!ht]
\begin{algorithmic}[1]
\ENSURE{[${\bfxi^{(1)}, \ldots, \bfxi^{(N)}}$] = \texttt{Method1}($\bfA$, $\bfR$, $\bfQ$, {$\bfG$}, $\bfb$, {$N$})}
\STATE {Use genHyBR to get $k$, $\bfs_k, \lambda, \bfV_k, \bfB_k$ (see \cref{ss_genGK})}
\STATE Compute Cholesky factorization $\bfM\t\bfM = \bfB_k\t\bfB_k$
\STATE Apply $\bfL^{-\top}$ to columns of $\bfV_k\bfM\t$ to get $\bfY_k$. \COMMENT{{Application of $\bfL^{-\top}$ is done using approach in \cref{ss_krylov} (see \cref{e_L}, for the definition of $\bfL$).}}
\STATE Compute [$\bfZ_k$, $\bfTheta_k$] = \texttt{Lowrank}($\bfY_k$)
\STATE Compute $\bfD_k = \bfI_k \pm (\bfI_k+\bfTheta_k)^{-1/2}$
\FOR {$j=1,\dots,N$ }
\STATE Draw sample $\bfepsilon^{(j)} \sim \mc{N}(\bfzero,\bfI)$
\STATE Compute $\bfz = \bfepsilon^{(j)} - \bfZ_k\bfD_k\bfZ_k\t\bfepsilon^{(j)}$
\STATE Compute ${\bfxi^{(j)}} = {\bfs_k +} \bfL^{-1}\bfz$
\ENDFOR
\end{algorithmic}
\caption{Method 1: Generates $N$ samples from {$\widehat\pi_\text{post}$}}
\label{alg_sampling1}
\end{algorithm}

\subsection{Method 2: {Sampling from $\pi_\text{post}$}}\label{ssec:method2}
The second approach we describe generates approximate samples from the exact posterior distribution. First, we rewrite the posterior covariance matrix as
{\[ \post = ({\lambda^2}\bfQ^{-1} + \bfH)^{-1} =\bfQ\bfF^{-1}\bfQ \qquad \bfF \equiv {\lambda^2}\bfQ+\bfQ\bfH\bfQ.\]}
We {define
\[  \bfS_{\bfF} \equiv \bfQ\bfF^{-1/2} \]}
such that $\post = \bfS_{\bfF}\bfS_{\bfF}\t.$ In this method, computing a factorization of $\post$ requires computing square roots with $\bfF$. Assume that we have a preconditioner $\bfG$ satisfying $\bfG\bfG\t \approx \bfF^{-1}$. Armed with this preconditioner, we have the following factorization
\[ \post = \bfS_{\bfF}\bfS_{\bfF}\t \qquad  \bfS_{\bfF} \equiv \bfQ \bfG^\top(\bfG\bfF\bfG^\top)^{-1/2}. \]
The application of the matrix $(\bfG\bfF\bfG^\top)^{-1/2}$ to a randomly drawn vector can be accomplished by the Lanczos approach described in \cref{ss_krylov}.

\begin{algorithm}[!ht]
\begin{algorithmic}[1]
  \ENSURE{[${\bfxi}$] = \texttt{Method2}($\bfA$, $\bfR$, $\bfQ$, $\bfG$, {$\bfs_\text{post}$})}
\STATE Draw sample $\bfepsilon \sim \mc{N}(\bfzero,\bfI)$
\STATE Compute $\bfz = \bfG\t(\bfG\bfF\bfG\t)^{-1/2}\bfepsilon$ using Lanczos approach in \cref{ss_krylov}
\STATE Compute ${\bfxi}= {\bfs_\text{post} + }\bfQ\bfz$
\end{algorithmic}
\caption{Method 2: Sampling from {$\pi_\text{post}$}}
\label{alg_sampling2}
\end{algorithm}

As currently described, computing approximate samples from $\post$ requires {computing $\bfs_\text{post}$ and} applying the matrix $\bfA$ and its adjoint $\bfA\t$. However, this may be computationally expensive for several problems of interest. {Here we use $\bfs_k$ as an approximation to $\bfs_\text{post}$.}
A variant of this method, not considered in this paper, follows by replacing the data-misfit part of the Hessian $\bfH$ by its {low-rank} approximation $\widehat\bfH$, defined in \cref{eqn:lowrank}. Define
\[  \widehat\bfF \equiv {\lambda^2}\bfQ+\bfQ\widehat\bfH\bfQ.\]
Therefore, we compute the following factorization of the approximate posterior covariance
\[ \posth = \widehat\bfS_\bfF\widehat\bfS_{\bfF}\t \qquad  \widehat\bfS_{\bfF} \equiv \bfQ \bfG^\top(\bfG\widehat\bfF\bfG^\top)^{-1/2} . \]

\subsection{Discussion}\label{ssec:disc}
We now compare the two proposed methods for generating approximate samples from the posterior. The first approach only uses the forward operator $\bfA$ in the precomputation phase to generate the low-rank approximation and subsequently uses the low-rank approximation as a surrogate. This can be computationally advantageous if the forward operator is very expensive or if many samples are desired. On the other hand, if accuracy is important {or only a few samples are needed,} then the second approach is recommended since it targets the full posterior distribution.

In Method 1, we generate samples from the approximate posterior distribution; the following result quantifies the error in the samples. Define {$\bfS = \bfQ^{1/2}(\lambda^{2}\bfI + \bfH_{\bfQ})^{-1/2}$ such that $\post = \bfS\bfS\t$ and let $\bfepsilon$ be a random draw from $\mathcal{N}(\bfzero,\bfI)$, then
$$\bfs = \spost + \bfS\bfepsilon \quad \mbox{and} \quad \widehat \bfs = \bfs_k + \widehat \bfS\bfepsilon$$
are samples from $\pi_{\rm post}$ and $\widehat \pi_{\rm post}$ respectively, where $\widehat\bfS$ is defined in \cref{eq:Shatdefinition}.}
\begin{theorem}\label{thm:sample}
Let $\posth$ be the approximate posterior covariance matrix generated by running $k$ steps of the gen-GK algorithm.   The error in the sample $\widehat\bfs$ satisfies
\[ {\|\bfs - \widehat\bfs\|_{\lambda^2\bfQ^{-1}} \leq \lambda^{-1}\left(\frac{\omega_k\alpha_1\beta_1}{\lambda^2 + \omega_k} + \sqrt{\frac{\lambda^2\omega_k}{\lambda^2+\omega_k}}\|\bfepsilon\|_{2}\right).} \]
\end{theorem}
\begin{proof} See \cref{sec:appendix3}. \end{proof}

\cref{thm:sample} states that if $\omega_k$ is sufficiently small, then the accuracy of the samples is high. The samples, thus generated, can then be used \textit{as is} in applications. Otherwise they can be used as candidate draws from a proposal distribution $\widehat\pi_\text{post}$. This proposal distribution can be used inside an independence sampler, similar to the approach in~\cite{brown2018low}.


\section{Numerical results}
\label{sec:numerics}
In~\cref{subsec:bounds}, we investigate the accuracy of the low-rank approximation to $\bfH$ and the subsequent bounds that were derived in \cref{sec:postapprox}. Then, in~\cref{subsec:sample_ex}, we describe our choice of preconditioners and demonstrate the efficiency of the preconditioned approaches proposed in \cref{sec:sampling} for generating samples from the posterior and approximate posterior.  In the final experiment provided in~\cref{subsec:dynamic}, we demonstrate our methods on a very large dynamic tomography reconstruction problem.

\subsection{Bounds for the posterior covariance matrix}
\label{subsec:bounds}
For this example, we use the \verb|heat| example from the Regularization Toolbox~\cite{hansen1994regularization}. Matrix $\bfA$ was $256\times 256$, and the observations were generated as~\cref{eq:linear problem}, where $\bfdelta$ models the observational error. In the experiments, we take $\bfdelta$ to be $1\%$ additive Gaussian white noise. We let $\bfQ$ be a
$256\times 256$ covariance matrix that was generated using an exponential kernel
$\kappa(r) = \exp(-r/\ell)$ where $r$ is the distance between two points and $\ell = 0.1$ is the correlation length. First, we use gen-HyBR to compute an approximate MAP estimate and simultaneously estimate a good regularization parameter. Using a weighted generalized cross validation (WGCV) method, the computed regularization parameter was $\lambda^2 \approx 5\times 10^3$. The regularization parameter was then fixed for the remainder of the experiment.

\begin{figure}[!ht]\centering
\includegraphics[scale=0.3]{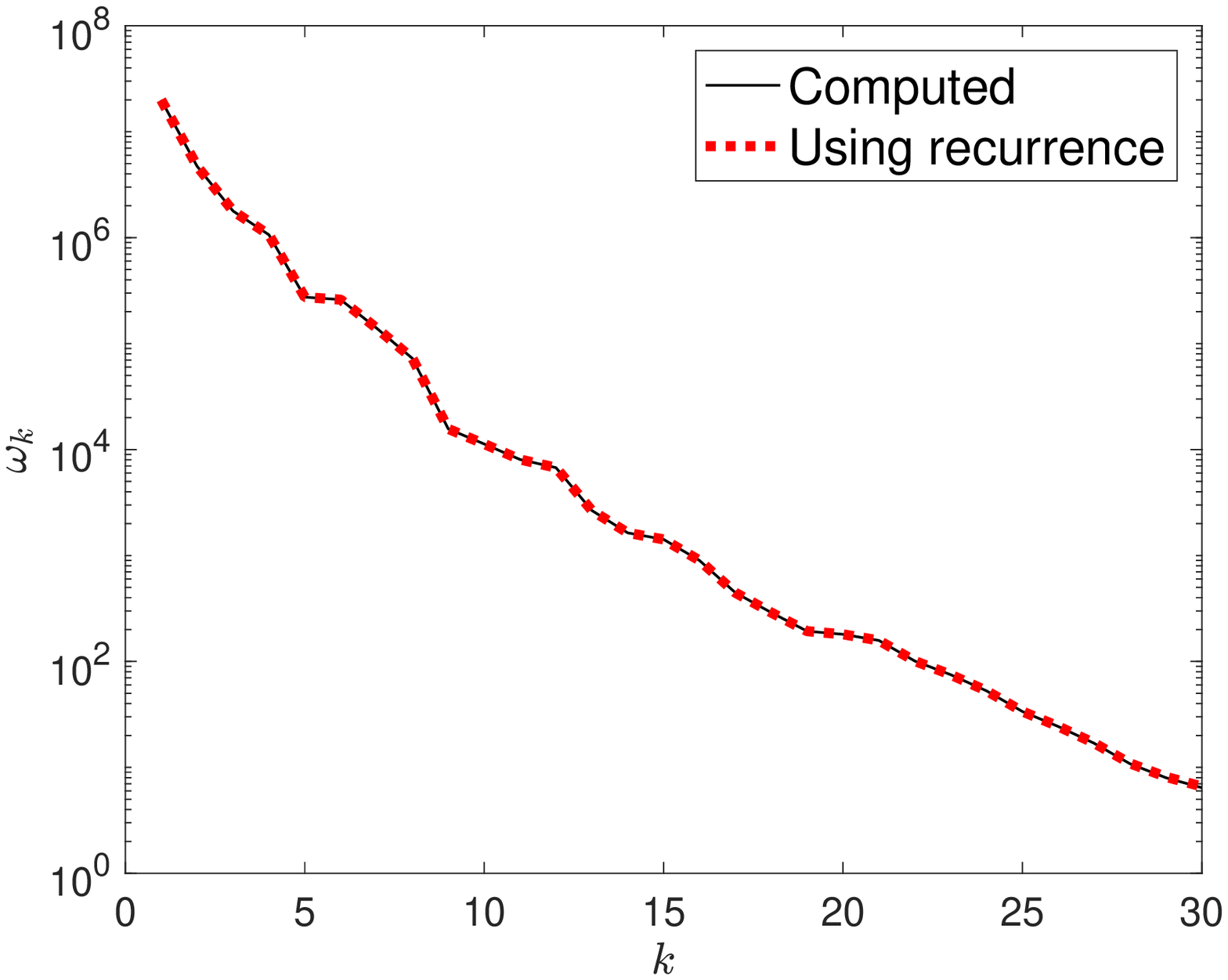}
\includegraphics[scale=0.3]{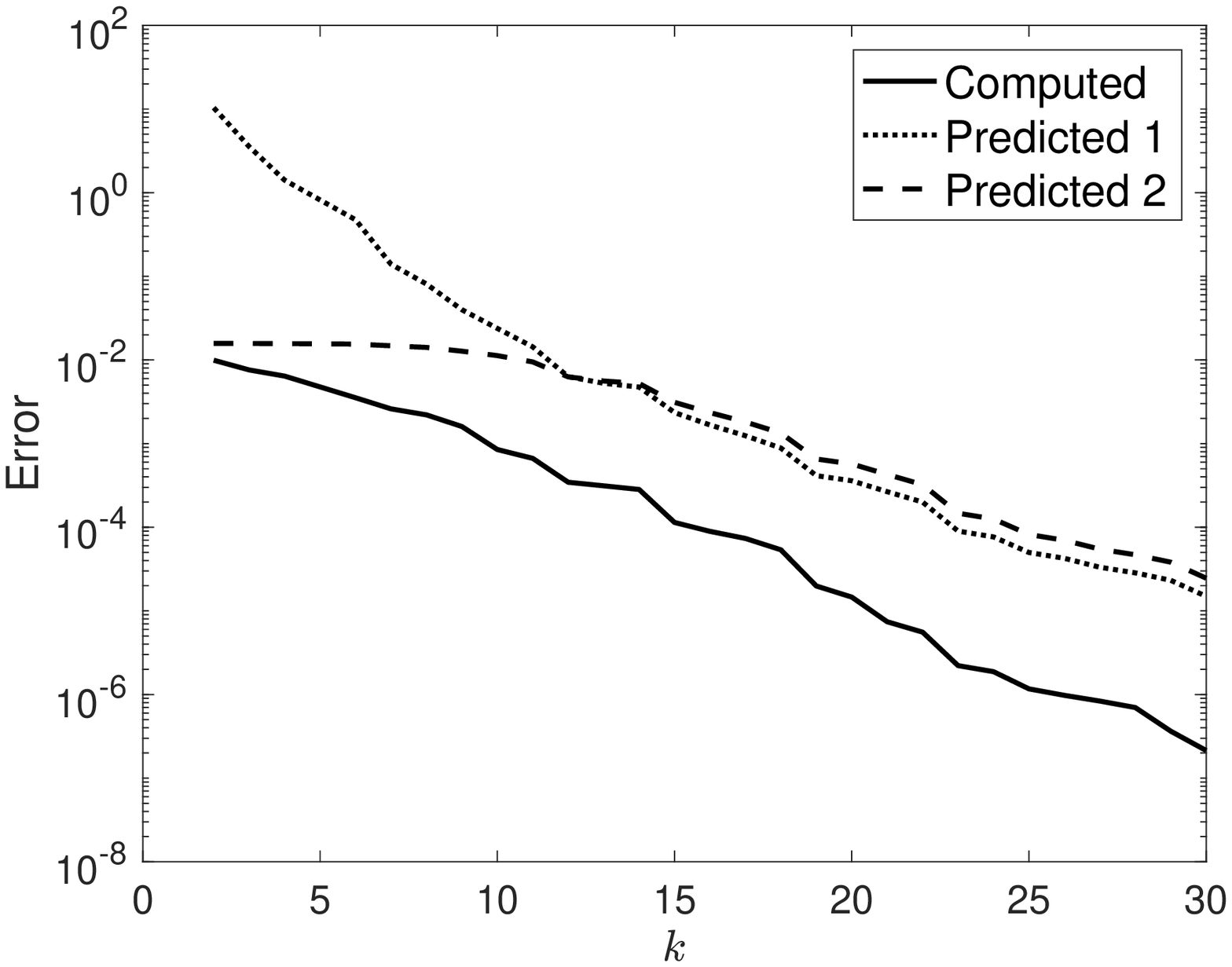}
\caption{The left plot contains computed values of $\omega_k$: the error between the true and the approximate prior-preconditioned Hessian for the data-misfit, as a function of the iteration $k$. The values for $\omega_k$, as computed by the recurrence relationship presented in \cref{prop:recur}, are provided in the dotted line. The right plot contains the errors for the posterior covariance matrix $\|\post - \posth\|_F$ as a function of the iteration, along with the two predicted bounds proposed in \cref{thm:post}.}
\label{fig:postacc}
\end{figure}

\cref{fig:postacc} shows the performance of the derived bounds. In the left plot, we track the accuracy of the prior-preconditioned data-misfit Hessian
$\omega_k = \|\bfH_{\bfQ} - \widehat{\bfH}_{\bfQ}\|_{F}$
 as a function of the number of iterations. The error {shows a sharp decrease with increasing number of iterations $k$}, and $\omega_k$ obtained by recursion is in close agreement with the actual error.
 This plot shows that, even in floating point arithmetic, the recursion relation for $\omega_k$ can be used to monitor the error of $\bfH_{\bfQ}$. The right plot in \cref{fig:postacc} contains the errors in the posterior covariance matrix $\|\post - \posth\|_F$, which also decreases {considerably}. We also provide both of the predicted bounds from \cref{thm:post}.
 While both bounds are qualitatively good, the first bound is slightly better at later iterations,
 whereas the second bound is more informative at earlier iterations. This can be attributed to
 the difference in the behavior of $\omega_k$ in the first bound versus  $\omega_k/(\lambda^2 + \omega_k)$ in the second bound. These plots provide evidence that the low-rank approximation $\widehat\bfH_\bfQ$ constructed using available components from the gen-GK bidiagonalization are quite accurate, and the bounds describing their behavior are informative.

\begin{figure}[!ht]\centering
\includegraphics[scale=0.4]{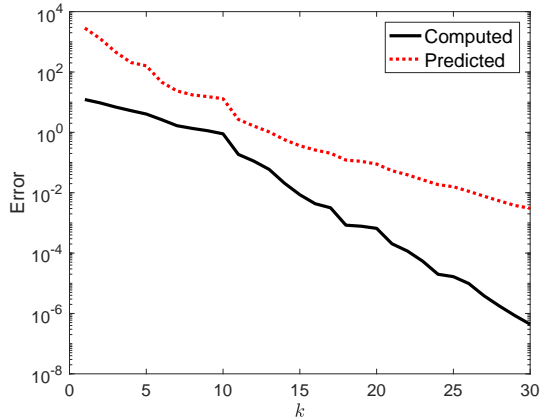}
\caption{This figure provides the computed error in the simplified KL divergence, along with the predicted bound, as a function of the iteration $k$.}
\label{fig:acckl}
\end{figure}

In the next illustration, we use the same problem setup, but we investigate {the} bound
for the KL divergence between the prior and the posterior distribution {\cref{thm:dkl}}.  We found that the bound for the quadratic term $\lambda^2 \, (\spost-\bfmu)\t\bfQ^{-1}(\spost-\bfmu)$ was too pessimistic, which
resulted in a large bound for the KL divergence in \cref{thm:dkl}.  Thus, we consider a simplified expression for the KL divergence,
\[D_\text{KL} = \frac12\left[ \trace ({\lambda^2}\bfQ^{-1}\post) -n - \log\det({\lambda^2}\bfQ^{-1}\post) \right] ,\]
such that the approximation is
\[ \widehat{D}_\text{KL} =  \frac12\left[- \trace ( \bfT_k(\bfT_k + \lambda^2\bfI )^{-1})  + \log\det(\bfI + \lambda^{-2}\bfT_k) \right]. \]
\cref{thm:dkl} then simplifies to $|D_\text{KL} - \widehat{D}_\text{KL}| \leq \lambda^{-2}\theta_k$, where $\theta_k$ is given in \cref{prop:trecur}. The error in the KL divergence is plotted in \cref{fig:acckl}, along with the corresponding bound. We see that that the bound captures the behavior of the KL divergence quite well. As for the quadratic term, we found empirically that the error decreases monotonically and is comparable to the simplified expression for the KL divergence. Even the pessimistic bound of \cref{thm:dkl} suggests that the error eventually decreases to zero with enough iterations. However, a more refined analysis is needed to develop informative bounds for the quadratic term and will be considered in future work. Future work could involve a tighter bound following the approach in \cite{golub2009matrices}.

\subsection{Sampling from the posterior}
\label{subsec:sample_ex}
After describing the choice of preconditioners, we show the performance of these preconditioners within Lanczos approaches for sampling from the prior and the posterior.

\subsubsection{Preconditioners for Mat\'ern Covariance Matrices}\label{ssec:precond}
In this experiment, we investigate preconditioned Lanczos methods described in~\cref{ss_krylov} for sampling from $\calN(\bfzero,\bfQ)$ where $\bfQ$ is defined by a Mat\'ern covariance kernel. We pick three covariance matrices $\bfQ$ corresponding to Mat\'ern parameters $\nu=1/2,\, 3/2,$ and $5/2$; this parameter controls the mean-squared differentiability of the underlying process. For a precise definition of the Mat\'ern covariance function, see~\cite[Equation (1)]{lindgren2011explicit}. The domain is set to $[0,1]^2$, and we choose a $300\times 300$ grid of evenly spaced points; thus, $\bfQ$ is a $90,000 \times 90,000$ matrix. The correlation length $\ell$ is $0.25$.

We use preconditioners of the form $\bfG = (-\bfDelta)^\gamma$ for parameters $\gamma \geq 1$, where $\bfDelta$ is the Laplacian operator discretized using the finite difference operator. These preconditioners are inspired by~\cite{lindgren2011explicit} and exploit the fact that integral operators based on Mat\'ern kernels have inverses that are fractional differential operators. We choose $\gamma = 1/2,\, 1,$ and $2$ corresponding to $\nu = 1/2,\, 3/2,$ and $5/2$ respectively.

\begin{figure}[!ht]\centering
\includegraphics[scale=0.45]{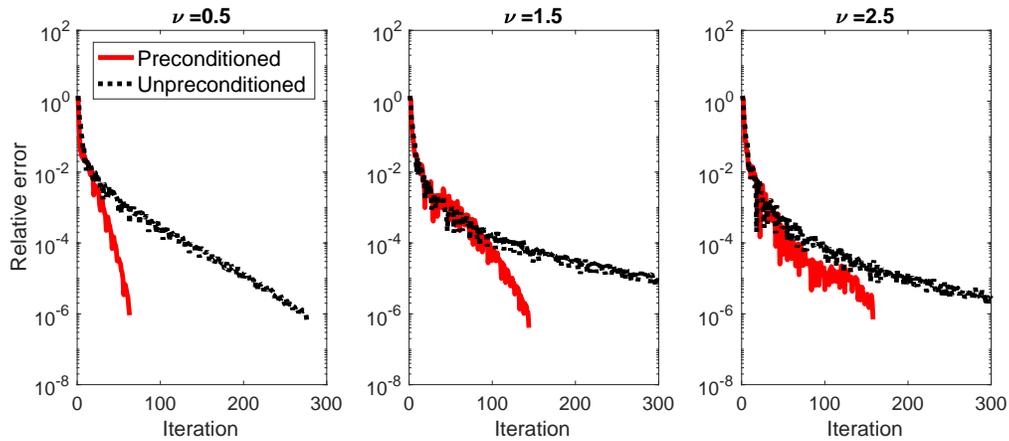}
\caption{Relative errors with and without preconditioning for sampling from $\calN(\bfzero,\bfQ)$. Preconditioners are based on fractional powers of the Laplacian $(-\bfDelta)^\gamma$.  The plots correspond to various choices of $\nu$ in the Mat\'ern covariance kernel and $\gamma$ in the preconditioner. (left) $\nu=1/2$ and $\gamma = 1/2$, (middle) $\nu = 3/2$ and $\gamma = 1$, and (right) $\nu = 5/2$  and $\gamma = 2$.}
\label{fig:maternconv}
\end{figure}

In \cref{fig:maternconv}, we provide the relative errors (computed as $\tilde\e_k$ from \cref{eq:ektilde}) per iteration of the preconditioned and unpreconditioned Lanczos approach. It is readily seen that for $\nu = 1/2$ and $3/2$, including the preconditioner can dramatically speed up the convergence. Some improvement is seen for the case of $\nu =5/2$, but the unpreconditioned solver does not converge within the maximum allotted number of iterations, which was set to $300$. Also, the number of iterations that it takes to converge increases with increasing parameter $\nu$; this is because the systems become more and more ill-conditioned for a fixed grid size. In summary, we see that integral powers of the Laplacian operator can be good preconditioners for sampling from priors with Mat\'ern covariance matrices.  Next we investigate the use of these preconditioners for efficient sampling from the posterior.

\subsubsection{Sampling from the posterior distribution} In this experiment, we use the \verb|PRspherical| test problem from the IRTools toolbox~\cite{gazzola2018ir,hansen2017air}.
The true image $\bfs$ and forward model matrix $\bfA$ that models spherical means tomography are provided.
We use the default settings provided by the toolbox; see~\cite{gazzola2018ir} for details. To simulate measurement error, we add $2\%$ additive Gaussian noise.

For a grid size of $128 \times 128$ and for $\bfQ$ that represents a Mat\'ern kernel with $\nu = 1/2$, we compute the MAP estimate using gen-HyBR and provide the reconstruction in the left panel of \cref{fig:pat2d}. The relative reconstruction error in the 2-norm was $0.0168$, and the regularization parameter determined using WGCV was $\lambda^2 \approx 19.48$. The regularization parameter was fixed for the remainder of this experiment. In \cref{fig:pat2d}, we also show a random draw from the prior distribution $\mathcal{N}(\bfzero,\lambda^{-2}\bfQ)$ in the middle panel and a random draw from the posterior distribution (computed using Method 2 in \cref{ssec:method2}) in the right panel. The same random vector $\bfepsilon \sim \mathcal{N}(\bfzero,\bfI)$ was used for both draws.

\begin{figure}[!ht]\centering
\includegraphics[width = \textwidth]{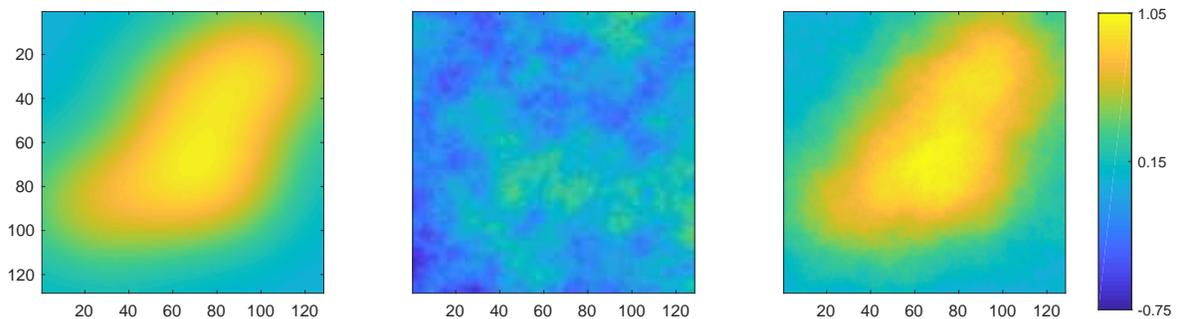}
\cprotect\caption{For the \texttt{PRspherical} problem, we provide the computed MAP estimate (left), a random draw from the prior distribution (middle), and a random draw from the posterior distribution (right).}
\label{fig:pat2d}
\end{figure}
Next we demonstrate the performance of Method 1 in \cref{ssec:method1} for sampling from the approximate posterior distribution $\widehat\pi_\text{post}$ and the performance of Method 2 in \cref{ssec:method2} for sampling from the posterior.  We vary the grid sizes from $16\times 16$ to $256\times 256$, and fix
all other parameters ($\nu = 1/2$, $2\%$ additive Gaussian noise) except the regularization parameter, which was determined for each problem using WGCV. The choice of preconditioners was described in \cref{ssec:precond}.

\begin{table}[!ht]\centering
\caption{For various examples of the \texttt{PRspherical} problem, we compare the performance of Methods 1 and 2 for sampling from the posterior.  The notation $n$ and $m$ is used to denote the number of unknowns and measurements respectively. For Method 1, we provide the number of genHyBR iterations required to compute the MAP estimate ($k$), the number of Lanczos iterations required for Step 3 of \cref{alg_sampling1} (Precompute), and the average number of iterations required for Step 9 of \cref{alg_sampling1} (Sampling). For Method 2, we provide the number of iterations (averaged over $10$ different runs) required for convergence in the preconditioned and unpreconditioned cases. }

\begin{tabular}{|l|c|c|c|c|c|c|}\hline
\multicolumn{2}{|c|}{} & \multicolumn{3}{c|}{Method 1} & \multicolumn{2}{c|}{Method 2} \\ \hline
$n$ & $m$   & $k$ & Precompute & Sampling  & Preconditioned & Unpreconditioned\\ \hline
$16\times 16 $ & $368$ & $52$ & $761$  &  $14.7$ & $22.0$ & $40.5$\\
$32\times 32 $ & $1,440$  & $32$ & $653$ & $21.1$ & $31.4$ & $69.5$\\
$64\times 64 $ & $5,824$  & $27$ & $749$ & $29.6$ & $44.0$ & $120.9$\\
$128\times 128 $ & $23,168$  &$37$ & $1433$ & $42.2$& $61.5$ & $210.1$ \\
$256\times 256 $ & $92,672$  & $63$ &  $3310$ & $60.2$  &  $85.6$ & $366$ \\ \hline
\end{tabular}
\label{tab:prsphericaltomo}
\end{table}

For Method 1, we use the gen-HyBR method to obtain the MAP estimate, the regularization parameter $\lambda^2$, and the low-rank approximation $\widehat\bfH_\bfQ$. In \cref{tab:prsphericaltomo} we report the number of genHyBR iterations as $k$; see \cite{chung2017generalized,chung2008weighted} for details on stopping criteria.
Then, we use \cref{alg_sampling1} to generate samples.
Notice that step 3 of \cref{alg_sampling1} requires the application of $\bfL^{-\top}$ to the low-rank approximation; this is accomplished by using the approach described in \cref{ss_krylov}, coupled with the choice of preconditioner described in \cref{ssec:precond}. The number of Lanczos iterations required for Step 3 is reported in the Precompute column of \cref{tab:prsphericaltomo}. Then, for each sample, step 9 of \cref{alg_sampling1} requires the application of $\bfL^{-1}$, which is also done using a Lanczos iterative process; the number of iterations for this step, averaged over 10 samples, is listed in the final column of \cref{tab:prsphericaltomo}.

For Method 2 we report the {average} number of iterations for the Lanczos solver to converge (i.e., achieving a residual tolerance of $10^{-6}$) with and without a preconditioner in \cref{tab:prsphericaltomo}.
We observe that the number of iterations required to achieve a desired tolerance increases with increasing problem size. This is expected since the number of measurements increases with increasing problem size, and the iterative solver has to work harder to process the additional ``information content.''
We also notice that including the preconditioner cuts the number of iterations roughly in half. For the largest problem we consider here, the unpreconditioned iterative solver required over four times the number of iterations as the preconditioned solver. Since each iteration requires one matvec with $\bfA$ and one with $\bfA\t$, each iteration can be quite expensive; the use of a preconditioner is beneficial in this case. Finally, another important observation is that although the preconditioners proposed in \cref{ssec:precond} were designed for the prior covariance matrix $\bfQ$, here they were used for the matrix $\bfF$ instead; nevertheless, the results in \cref{tab:prsphericaltomo} demonstrate that the preconditioners were similarly effective.

We make a few remarks about the results. First, the precomputation step to generate the low-rank approximation {in Method 1} requires a considerable number of matvecs involving $\bfQ$ but far fewer involving $\bfA$.  Next, the number of iterations required for generating the samples {in Method 1} is, on average, smaller than those reported for Method 2 for comparable problem size. The reason for this is that the preconditioner is designed for $\bfQ$ rather than $\bfF$.

\subsection{Dynamic Tomography Example}
\label{subsec:dynamic}
In this experiment, we consider a dynamic tomography setup where the goal is to reconstruct a sequence of images from a sequence of projection datasets.  {Such scenarios are common in dynamic photoacoustic or dynamic electrical impedance tomography, where the underlying parameters change during the data acquisition process \cite{wang2014fast,schmitt2002efficient2,hahn2014efficient}. Reconstruction is particularly challenging for nonlinear or nonparametric deformations and often requires including a spatiotemporal prior \cite{chung2018efficient,schmitt2002efficient1}.}

For this example, the true images were generated using two Gaussians moving in different directions in the image domain.  We consider a sequence of $20$ images (e.g., time points), where each image is $256\times 256$. In \cref{fig:tomo}, we provide $5$ of the true images.
\begin{figure}[!ht]\centering
\includegraphics[scale=0.75]{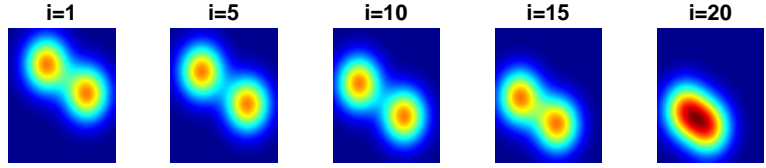}
\caption{$5$ of the $20$ true images for dynamic tomography example are provided.}
\label{fig:tomo}
\end{figure}
We consider a linear problem of the form \cref{eq:linear problem}, where
\begin{equation}
  \bfs = \begin{bmatrix}
    \bfs^{(1)}\\ \vdots \\ \bfs^{(20)}
  \end{bmatrix} \in \bbR^{20*256^2}, \quad
  \bfA = \begin{bmatrix}
    \bfA^{(1)} & & \\
 &\ddots &\\
 &&\bfA^{(20)}
\end{bmatrix}, \quad \mbox{and} \quad
\bfd = \begin{bmatrix}
  \bfd^{(1)}\\ \vdots \\ \bfd^{(20)}
\end{bmatrix},
\end{equation}
where $\bfA^{(i)} \in \bbR^{{18*362} \times 256^2}$
 represents a spherical projection matrix corresponding to $18$ equally spaced angles between $i$ and $340+i$ for $i=1, \ldots, 20,$ and $\bfd^{(i)} \in \bbR^{18*362}$ contains projection data. {To simulate measurement error we add $2\%$ Gaussian noise.}

For the spatiotemporal prior, we let $\bfQ = \bfQ_t \kron \bfQ_s,$ where $\bfQ_t\in \bbR^{20 \times 20}$ and $\bfQ_s\in \bbR^{256^2 \times 256^2}$ correspond to Mat\'{e}rn kernels with $\nu = 2.5, \ell = 0.1$ and $\nu = 0.5, \ell = 0.25$ respectively.  First we use the generalized hybrid approach from \cite{chung2017generalized} to compute an approximation of the MAP estimate and to determine $\lambda$ using WGCV.  In \cref{fig:dynamic_MAP} we provide $5$ of the images from the MAP reconstruction.
\begin{figure}[!ht]\centering
\includegraphics[scale=0.75]{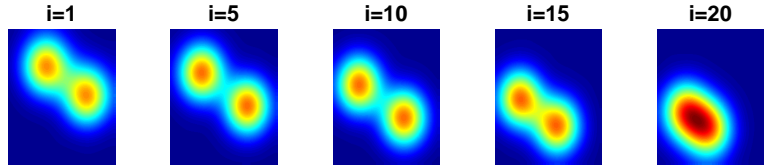}
\caption{Reconstructions from the MAP estimate using genHyBR for the dynamic tomography problem.}
\label{fig:dynamic_MAP}
\end{figure}

Since we can easily obtain a Cholesky factorization of $\bfQ_t^{-1} = \bfG_t\t \bfG_t,$ we define a preconditioner of the form $\bfG = \bfG_t \kron \bfG_s$ where $\bfG_s = (-\Delta)^\gamma$, the exponent $\gamma = 0.5$ and $\kron$ represents the Kronecker product.  Then we use the preconditioned sampling methods described in \cref{sec:sampling} to generate $10$ samples from the prior, the {approximate posterior (Method 1), and the posterior (Method 2)}.  Note that each sample is a $256\times 256\times 20$ volume.  In \cref{fig:dynamic_samples}, we select one sample and provide $5$ slices.
\begin{figure}[!ht]\centering
\includegraphics[scale=0.95]{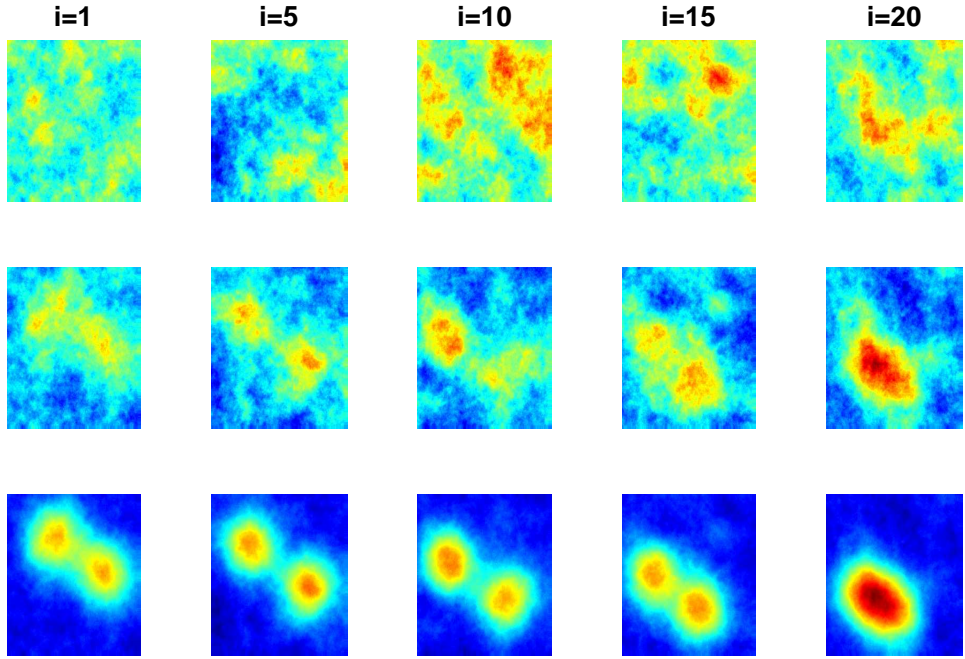}
\caption{A random sample from the prior (first row), {the approximate posterior (middle row), and the posterior (bottom row)} for the dynamic tomography problem.}
\label{fig:dynamic_samples}
\end{figure}

Next we compare CPU timings (in seconds) and number of iterations, averaged over $10$ samples, for both the preconditioned and unpreconditioned versions. In \cref{tab:dynamic_time}, we provide timings and iteration counts in parentheses for generating a sample from the prior, the approximate posterior, and the posterior.  For sampling from the approximate posterior, we also provide the number of Lanczos iterations for precomputation followed by the average number of iterations for sampling (similar to \cref{tab:prsphericaltomo}).  Again, we provide results for various problem dimensions.
\begin{table}[!ht]\centering
\caption{A comparison of CPU timings in seconds ($s$) and iteration counts for the dynamic tomography problem for both the preconditioned and unpreconditioned cases.  Timings are averaged over $10$ samples and average iteration counts are provided in parentheses, with iteration counts separated into precomputation and sampling for the approximate posterior sample.}
\begin{tabular}{l|c|c}
  \hline
$64 \times 64 \times 20$  & Preconditioner (iter)  &  No preconditioner (iter) \\ \hline
Prior Sample  & 1.56 $s$ (31.4) & 297.54 $s$ (500+)  \\ \hline
Approximate Posterior Sample & 8.00 $s$ (1379, 31) & 1013.84 $s$ (16143, 500+) \\ \hline
Posterior Sample  & 15.55 $s$ (104.5) &  312.54 $s$ (500+)\\ \hline
  \multicolumn{3}{c}{ } \\ \hline
$128 \times 128 \times 20$  & Preconditioner (iter)  &  No preconditioner (iter) \\ \hline
Prior Sample  & 15.37 $s$ (52.3) & 1182.51 $s$ (500+)  \\ \hline
Approximate Posterior Sample & 58.02 $s$ (1850, 48.9) & 4984.79 $s$ (18247, 500+) \\ \hline
Posterior Sample  & 140.11 $s$ (158.2) &  1246.74 $s$ (500+)\\ \hline
\multicolumn{3}{c}{ } \\ \hline
$256 \times 256 \times 20$    & Preconditioner (iter)  &  No preconditioner (iter) \\ \hline
Prior Sample  & 137.81 $s$ (75) & 5358.00 $s$ (500+)  \\ \hline
Approximate Posterior Sample & 460.29 $s$ (2381, 75) &  23147.48 $s$ (18494, 500+) \\ \hline
Posterior Sample  & 1255.31 $s$ (238.4) &  5563.96 $s$ (500+)\\ \hline
\end{tabular}
\label{tab:dynamic_time}
\end{table}
We remark that sampling from the approximate posterior requires an upfront cost from precomputation, but if many samples are required, that cost can be amortized.  On the other hand, if we need only a few, more accurate samples, then sampling from the true posterior may be more efficient. We also observe that the use of a preconditioner significantly cuts the number of required iterations. Indeed, none of the unpreconditioned iterative solvers considered for this example converged within the maximum number of iterations taken to be $500$.


\section{Conclusions}
\label{sec:conclusions}

This paper considers the challenging problem of providing an efficient representation for the posterior covariance matrix arising in high-dimensional inverse problems. To this end, we propose an approximation to the posterior covariance matrix as a low-rank perturbation of the prior covariance matrix. The approximation is computed using information from the
{gen-GK bidiagonalization}
generated while computing the MAP estimate. As a result, we obtain an approximate and efficient representation for ``free.'' Several results are presented to quantify the accuracy of this representation and of the resulting posterior distribution. We also show how to efficiently compute measures of uncertainty involving the posterior distribution. Then we present two variants that utilize a preconditioned Lanczos solver to efficiently generate samples from the posterior distribution. The first approach generates samples from an approximate posterior distribution, whereas the second approach generates samples from the exact posterior distribution. The approximate samples can be used \textit{as is} or as candidate draws from a proposal distribution that closely approximates the exact posterior distribution.

There are several avenues for further research. The first important question is: Can we replace the bounds in the Frobenius norm by the spectral norm? The reason we employed the Frobenius norm is because of the recurrence relation in \cref{prop:recur}. Another issue worth exploring is if we can give bounds for the error in the low-rank approximation $\omega_k$ explicitly in terms of the eigenvalues of $\bfH_{\bfQ}$. This can be beneficial for deciding \text{a priori} the number of iterations required for an accurate low-rank approximation when the rate of decay of eigenvalues of $\bfH_{\bfQ}$ is known. Finally, we are interested in exploring {the use of} the approximate  posterior distribution as a surrogate for the exact posterior distribution inside a Markov Chain Monte Carlo (MCMC) sampler. This is of particular interest for nonlinear problems where the posterior distribution is non-Gaussian.  MCMC methods rely heavily on the availability of a good proposal distribution. One approach is to linearize the forward operator about the MAP estimate (the so-called Laplace's approximation) resulting in a Gaussian distribution with similar structure to $\pi_\text{post}$. This approximation to the true posterior distribution can be used as a proposal distribution, see for e.g.~\cite{martin2012stochastic,petra2014computational}.


\section{Acknowledgements} This work was partially supported by NSF DMS 1720398 (A. Saibaba and K. Petroske), NSF DMS 1654175 (J.~Chung), and NSF DMS 1723005 (J.~Chung). The authors would like to thank Silvia Gazzola, Per Christian Hansen, and James Nagy for generously sharing an advanced copy of their preprint~\cite{gazzola2018ir} and code that we used in our numerical experiments.


\appendix
\section{Proofs}
\subsection{Derivation of \cref{eqn:meanhat}}
\label{sec:appendix4}
{
First, we plug in $\posth= (\lambda^2 \bfQ^{-1} + \widehat\bfH)^{-1}$ and rearrange to get
\[\begin{aligned}
  \posth \bfA\t\bfR^{-1}\bfb & = (\lambda^2 \bfQ^{-1} + \bfV_k \bfT_k\bfV_k\t)^{-1}\bfA\t \bfR^{-1}\bfb\\
  & = (\lambda^2 \bfI+ \bfQ \bfV_k \bfT_k\bfV_k\t)^{-1} \bfQ \bfA\t \bfR^{-1}\bfb\,.
\end{aligned}\]
Then, using the gen-GK relationships, we note that $$\bfA\t \bfR^{-1}\bfb = \bfA\t \bfR^{-1} \bfU_{k+1}\beta_1 \bfe_1= \bfV_k \bfB_k\t \beta_1 \bfe_1\,. $$
Furthermore, using the {Woodbury formula~\cite[Equation (0.7.4.1)]{horn2012matrix}}, we have
$$ (\lambda^2 \bfI+ \bfQ \bfV_k \bfT_k\bfV_k\t)^{-1} = \lambda^{-2} \bfI - \lambda^{-4} \bfQ \bfV_k (\bfT_k^{-1} + \lambda^{-2} \bfI)^{-1}\bfV_k\t\,.$$
Thus, we get
\[\begin{aligned}
  \posth \bfA\t\bfR^{-1}\bfb & = \left(\lambda^{-2} \bfI - \lambda^{-4} \bfQ \bfV_k \left(\bfT_k^{-1} + \lambda^{-2} \bfI\right)^{-1}\bfV_k\t\right) \bfQ \bfV_k \bfB_k\t \beta_1 \bfe_1\\
  & = \bfQ \bfV_k \left(\lambda^{-2} \bfI - \lambda^{-4} \left(\bfT_k^{-1} + \lambda^{-2} \bfI \right)^{-1} \right) \bfB_k\t \beta_1 \bfe_1 \\
  & = \bfQ \bfV_k (\bfT_k + \lambda^{2} \bfI)^{-1} \bfB_k\t \beta_1 \bfe_1\,,
\end{aligned}\]
where the last equality uses the fact that
$(\bfT_k^{-1} + \lambda^{-2} \bfI)^{-1}  = \lambda^2 \bfI - \lambda^4 (\bfT_k + \lambda^{2} \bfI)^{-1}.$
Since $\bfT_k = \bfB_k\t\bfB_k,$ we have the desired result.
}

\subsection{Proofs {for} \cref{ssec:postacc}}
\label{sec:appendix1}
\begin{proof}[\cref{prop:recur}]
First, we recognize that
{$\widehat\bfH =  \bfV_k\bfT_k\bfV_k\t$, where $\bfT_k = \bfV_k\t\bfQ\bfH\bfQ\bfV_k$ is a tridiagonal matrix of the form}
\[\bfT_k = \begin{bmatrix}\mu_1 & \nu_2 \\ \nu_2 & \mu_2 & \nu_3 \\ & \ddots & \ddots & \ddots \\ & & \nu_{k-1} & \mu_{k-1} & \nu_k \\ & && \nu_k & \mu_k \end{bmatrix}, \]
where $\mu_j = \alpha_j^2 + \beta_{j+1}^2$ and $\nu_j = \alpha_j\beta_j$ for $j=1,\dots,k$.

For simplicity denote $\widehat{\bfV}_k = \bfQ^{1/2}\bfV_k$ and note that the columns of $\widehat{\bfV}_k$ are orthonormal. Then write
\[{\bfH_\bfQ - \widehat\bfH_\bfQ = }
 (\bfI - \widehat\bfV_k{\widehat\bfV_k\t})\bfH_{\bfQ} +  \widehat\bfV_k\widehat\bfV_k\t \bfH_{\bfQ}(\bfI - \widehat\bfV_k\widehat\bfV_k\t).\]
The observation that $ (\bfI - \widehat\bfV_k\widehat\bfV_k\t)\bfH_{\bfQ} \perp \widehat\bfV_k\widehat\bfV_k\t\bfH_{\bfQ}(\bfI - \widehat\bfV_k\widehat\bfV_k\t)$ with respect to the trace inner product, it is easy to show that
\[ \omega_k^2 = \| (\bfI - \widehat\bfV_k\widehat\bfV_k\t)\bfH_{\bfQ} \|_{F}^2 +\| \widehat\bfV_k\widehat\bfV_k\t \bfH_{\bfQ}(\bfI - \widehat\bfV_k\widehat\bfV_k\t)\|_{F}^2. \]
The second term is easy since using the gen-GK relationships, we have
	\[ \widehat\bfV_k\widehat\bfV_k\t \bfH_{\bfQ}(\bfI - \widehat\bfV_k\widehat\bfV_k\t) = \alpha_{k+1}\beta_{k+1} \widehat\bfv_k\widehat\bfv_{k+1}\t,\]
	and thus
	$\|\alpha_{k+1}\beta_{k+1} \widehat\bfv_k\widehat\bfv_{k+1}\t\|_{F}^2  = |\alpha_{k+1}\beta_{k+1}|^2$. For the first term, we denote $\eta_k = \|(\bfI - \widehat\bfV_k\widehat\bfV_k\t)\bfH_{\bfQ}\|_F$, so that
\begin{equation}\label{eqn:inter1}
\omega_{k}^2 = \eta_k^2 + |\alpha_{k+1}\beta_{k+1}|^2.
\end{equation}
Then write $\bfI - \widehat\bfV_k\widehat\bfV_k\t =\bfI - \widehat\bfV_{k+1}\widehat\bfV_{k+1}\t  + \widehat\bfv_{k+1}\widehat\bfv_{k+1}\t $ and again apply Pythagoras' theorem to get
\[ \eta_k^2 = \eta_{k+1}^2 + \| \widehat\bfv_{k+1}\widehat\bfv_{k+1}\t\bfH_{\bfQ}\|_F^2.\]
From the {gen-GK} relations, it can be verified that
\begin{equation}\label{eqn:inter2}
\begin{aligned}\bfH_{\bfQ}\widehat\bfv_{k+1}\widehat\bfv_{k+1}\t = & \> \alpha_{k+1}\beta_{k+1}\widehat\bfv_k\widehat\bfv_{k+1}\t + (\alpha_{k+1}^2 + \beta_{k+2}^2) \widehat\bfv_{k+1}\widehat\bfv_{k+1}\t\\
& \> \qquad \qquad +\alpha_{k+2}\beta_{k+2}\widehat\bfv_{k+2}\widehat\bfv_{k+1}\t.
\end{aligned}
\end{equation}
Since each term is mutually orthogonal, this implies
\[ \eta_k^2 = \eta_{k+1}^2 + |\alpha_{k+1}\beta_{k+1}|^2 + |\alpha_{k+1}^2 + \beta_{k+2}^2|^2 + |\alpha_{k+2}\beta_{k+2}|^2.\]
Together with \cref{eqn:inter1}, {we get the desired recurrence.}
\end{proof}

\begin{proof}[\cref{thm:post}]

We now consider the error in the posterior covariance matrix. For the first bound, using
$$\post = \bfQ^{1/2}(\lambda^2\bfI + \bfH_{\bfQ})^{-1}\bfQ^{1/2}\,,
 $$ we have
\[\begin{aligned}
  \|\post - \posth\|_F & \leq  \|\bfQ\|_2 \|(\lambda^{2}\bfI + \bfH_{\bfQ})^{-1} - (\lambda^{2}\bfI + \widehat\bfH_{\bfQ})^{-1}\|_F\\
  & = \lambda^{-2} \|\bfQ\|_2 \|(\bfI + \lambda^{-2}\bfH_{\bfQ})^{-1} - (\bfI +\lambda^{-2} \widehat\bfH_{\bfQ})^{-1}\|_F.
\end{aligned}\]
With $f(x) = x/(1+x)$, it is verifiable that
\[  (\bfI + \lambda^{-2}\bfH_{\bfQ})^{-1} - (\bfI + \lambda^{-2}\widehat\bfH_{\bfQ})^{-1} = f( \lambda^{-2}\widehat\bfH_{\bfQ}) - f(\lambda^{-2}\bfH_{\bfQ}).\]
The function $f$ is operator monotone~\cite[Proposition V.1.6]{bhatia2013matrix} and satisfies $f(0) = 0$. Since both $\lambda^{-2}\bfH_{\bfQ}$ and $\lambda^{-2}\widehat\bfH_{\bfQ}$ are positive semi-definite, using~\cite[Theorem X.1.3]{bhatia2013matrix}, we obtain
\[   \|(\bfI + \lambda^{-2}\bfH_{\bfQ})^{-1} - (\bfI + \lambda^{-2}\widehat\bfH_{\bfQ})^{-1}\|_F \leq \| |\bfE| (\bfI + |\bfE|)^{-1}\|_F,\]
where we let $\bfE = \lambda^{-2}(\bfH_{\bfQ} - \widehat\bfH_{\bfQ})$, and $|\bfE | = (\bfE^*\bfE)^{1/2}. $ Note that both $|\bfE|$ and $\bfE$ have the same singular values, so $\||\bfE|\|_F = \|\bfE\|_F$. Since $|\bfE|$ is positive semi-definite,  the singular values of $ (\bfI+|\bfE|)^{-1}$ are at most $1$. By submultiplicativity inequality and $\||\bfE|\|_F = \|\bfE\|_F$, we have
\begin{equation}
  \label{appendix:bound}
   \|(\lambda^{2}\bfI + \bfH_{\bfQ})^{-1} - (\lambda^{2}\bfI + \widehat\bfH_{\bfQ})^{-1}\|_F \leq   \| \lambda^{-2} (\bfH_\bfQ- \widehat\bfH_\bfQ) \|_F = \lambda^{-2} \omega_k
\end{equation}
and hence the desired result:
\begin{equation}
  \|\post - \posth\|_F  \leq  \lambda^{-2} \|\bfQ\|_2 \|\lambda^{-2} (\bfH_\bfQ - \widehat \bfH_\bfQ)\|_F = \lambda^{-4} \omega_k  \|\bfQ\|_2.
\end{equation}
For the second bound, we reserve the use of spectral and Frobenius norms
\[
  \|\post - \posth\|_F  \leq  \lambda^{-2} \|\bfQ\|_F \|(\bfI + \lambda^{-2}\bfH_{\bfQ})^{-1} - (\bfI + \lambda^{-2}\widehat\bfH_{\bfQ})^{-1}\|_2.\]
  Again, let $\bfE = \lambda^{-2}(\bfH_{\bfQ} - \widehat\bfH_{\bfQ})$, and use~\cite[Theoerem X.1.1]{bhatia2013matrix} with $f(x) = x/(1+x)$, to obtain
\[   \|(\bfI + \lambda^{-2}\bfH_{\bfQ})^{-1} - (\bfI + \lambda^{-2}\widehat\bfH_{\bfQ})^{-1}\|_2 \leq \frac{\|\bfE\|_2}{1 + \|\bfE\|_2}.\]
It is readily verified that if $0 \leq a \leq b$, then ${a}(1 + a)^{-1} \leq {b}(1+b)^{-1}$, and so
\begin{equation}\label{eqn:invpert} \|(\bfI + \lambda^{-2}\bfH_{\bfQ})^{-1} - (\bfI + \lambda^{-2}\widehat\bfH_{\bfQ})^{-1}\|_2 \leq \frac{\|\bfE\|_2}{1 + \|\bfE\|_2}\leq \frac{\|\bfE\|_F}{1 + \|\bfE\|_F} = \frac{\omega_k}{\lambda^2 + \omega_k}.\end{equation}
The recognition that $\|\bfE\|_F = \lambda^{-2}\omega_k$ completes the proof.
\end{proof}

\subsection{Lemma of independent interest}
We will need the following lemma to prove \cref{thm:dklacc,thm:dkl}. This may be of independent interest beyond this paper.
\begin{lemma}\label{lemma:trdet} Let $\bfA \in \mb{R}^{n\times n}$ be symmetric positive semidefinite
  and let $\bfP\in \mb{R}^{n\times n}$ be an orthogonal projection matrix. Then the following results hold
\[\begin{aligned}
|\trace(\bfI + \bfA )^{-1} -\trace(\bfI + \bfP\bfA\bfP)^{-1}| \leq & \> \trace(\bfA-\bfP\bfA\bfP),\\
 \trace\left[(\bfI + \bfA)(\bfI + \bfP\bfA\bfP)^{-1}\right] \leq & \> n + \trace(\bfA-\bfP\bfA\bfP) \\
 0 \leq \log\det(\bfI + \bfA) -\log\det(\bfI + \bfP\bfA\bfP)  \leq & \> \trace(\bfA-\bfP\bfA\bfP).
\end{aligned}\]
\end{lemma}
\begin{proof}
Let $\{ \lambda_i \}_{i=1}^n$ and $\{\mu_i\}_{i=1}^n$ denote the eigenvalues of $\bfA$ and $\bfP\bfA\bfP$. Since both matrices are positive semidefinite, their eigenvalues are non-negative. Since $\bfP$ is a projection matrix, {its} singular values are at most $1$. The multiplicative singular value inequalities~\cite[Problem III.6.2]{bhatia2013matrix} say $\sigma_i(\bfP\bfA^{1/2}) \leq \sigma_i(\bfA^{1/2})$, so $\lambda_i \geq \mu_i$ for $i=1,\dots,n$, and therefore, $\trace(\bfA) \geq \trace(\bfP\bfA\bfP)$.
Then for the first inequality
\begin{align*}
 | \trace(\bfI + \bfP\bfA\bfP)^{-1} - \trace(\bfI + \bfA )^{-1}| = & \>  \left|\sum_{i=1}^n \frac{\lambda_i-\mu_i}{(1+\mu_i)(1+\lambda_i)} \right| \\
\leq & \> \left|\sum_{i=1}^n (\lambda_i-\mu_i)\right| = |\trace(\bfA -\bfP\bfA\bfP)|.
\end{align*}
The inequalities follow since $\lambda_i,\mu_i$ are nonegative. The absolute value disappears since $\trace(\bfA) \geq \trace(\bfP\bfA\bfP)$.

For the second inequality, write
\[ (\bfI + \bfA)(\bfI+\bfP\bfA\bfP)^{-1}  = \bfA (\bfI+\bfP\bfA\bfP)^{-1}-\bfP\bfA\bfP(\bfI+\bfP\bfA\bfP)^{-1} + \bfI.\]
Both $\bfA$ and $(\bfI+\bfP\bfA\bfP)^{-1}$ are positive semidefinite (the second matrix is definite), so the trace of their product is nonnegative~\cite[Exercise 7.2.26]{horn2012matrix}. Then a straightforward application of the von Neumann trace theorem~\cite[Theorem 7.4.1.1]{horn2012matrix} {leads to}
\[ \trace(\bfA(\bfI+\bfP\bfA\bfP)^{-1}) \leq \sum_{i=1}^n \frac{\lambda_i}{1+\mu_i}. \]
By utilizing its eigendecomposition, we see that $\trace[\bfP\bfA\bfP(\bfI+\bfP\bfA\bfP)^{-1}] = \sum_{i=1}^n \frac{\mu_i}{1+\mu_i}$. Putting it together{, we get}
\[\begin{aligned}
\trace[(\bfI + \bfA)(\bfI+\bfP\bfA\bfP)^{-1}] \leq &  \> n  + \sum_{i=1}^n \left( \frac{\lambda_i}{1+\mu_i} - \frac{\mu_i}{1+\mu_i} \right) \\
\leq & \> n + \sum_{i=1}^n\frac{\lambda_i-\mu_i}{1+\mu_i} \> \leq  \> n + \sum_{i=1}^n(\lambda_i-\mu_i).
\end{aligned}\]
Connecting the sum of the eigenvalues with the trace delivers the desired result.

For the third inequality, use Sylvester's determinant {identity~\cite[Corollary 2.11]{ouellette1981schur}} to write
\[ \log\det(\bfI + \bfP\bfA\bfP) = \log\det(\bfI + \bfA^{1/2}\bfP\bfA^{1/2}). \]
{Denote $\bfB = \bfA^{1/2}\bfP\bfA^{1/2}$ and introduce the notation of Loewner partial ordering~\cite[Section 7.7]{horn2012matrix}. Let $\bfM,\bfN \in \mb{R}^{n\times n}$ be symmetric. Then, $\bfM \preceq \bfN$ means $\bfN-\bfM$ is positive semidefinite. Since $\bfP \preceq \bfI$, it follows that $\bfB \preceq \bfA$~\cite[Theorem 7.7.2]{horn2012matrix}.} Then apply~\cite[Lemma 9]{alexanderian2017efficient}, to obtain
\[ 0 \leq \log\det(\bfI + \bfA) -\log\det(\bfI + \bfB) \leq \log\det(\bfI + \bfA- \bfB).\]

Finally since $\log(1+x) \leq x$ for $x \geq 0$, $\log\det(\bfI + \bfA- \bfB) \leq \trace( \bfA- \bfB)$. The proof is {completed by observing that} $\trace(\bfB) = \trace(\bfP\bfA\bfP)$ by the cyclic property of trace.
\end{proof}

\subsection{Proofs of \cref{ssec:acc} and \cref{ssec:info}}
\label{sec:appendix2}

\begin{proof}[\cref{prop:trecur}]
The linearity and cyclic property of trace estimator implies
\[ \theta_k = \trace( (\bfI- \widehat\bfV_k\widehat\bfV_k\t)\bfH_{\bfQ} ). \]
As in the proof of Proposition~\ref{prop:recur}, write $\bfI - \widehat\bfV_k\widehat\bfV_k\t =\bfI - \widehat\bfV_{k+1}\widehat\bfV_{k+1}\t  + \widehat\bfv_{k+1}\widehat\bfv_{k+1}\t $, so that
\[ \theta_k = \theta_{k+1} + \trace(\widehat\bfv_{k+1}\widehat\bfv_{k+1}\t\bfH_{\bfQ}).\]
The proof is finished if we apply the trace to the right hand side of \cref{eqn:inter2}.
\end{proof}

\begin{proof}[\cref{thm:dklacc}]
The lower bound follows from the property of the KL divergence and the fact that the distributions are not degenerate. The proof for the upper bound begins by providing an alternate expression for the error in the KL divergence.
\[ D_{KL}(\widehat\pi_\text{post} \| \pi_\text{post}) = \frac12 \left[\mc{E}_1 + \mc{E}_2 + \mc{E}_3\right],\]
where $\mc{E}_1 =   \trace(\posth\post^{-1}) - n$,
\[
\mc{E}_2=  \log\det(\post) - \log\det(\posth),  \quad \text{and} \quad
\mc{E}_3 =  \| \spost -{\bfs_k}\|_{\post^{-1}}^2.
\]
We tackle each term individually. The second term $\mc{E}_2$ simplifies since
\[\log\det(\post) - \log\det(\posth)  = \log\det(\bfI + \lambda^{-2} \widehat\bfH_\bfQ) - \log\det(\bfI + \lambda^{-2} \bfH_\bfQ). \]
Let $\bfM = \lambda^{-2}(\bfH_{\bfQ}) $, then with $\bfP = \widehat\bfV_k\widehat\bfV_k\t$ we have $\lambda^{-2}\widehat\bfH_\bfQ = \bfP\bfM\bfP$. Apply the third inequality in \cref{lemma:trdet} to conclude $\mc{E}_2 \leq 0$. For the first term $\mc{E}_1$, apply the second part of \cref{lemma:trdet} to obtain
\[\begin{aligned}
\trace(\posth\post^{-1}) = & \> \trace[(\bfI+ \lambda^{-2}\widehat\bfH_\bfQ)^{-1}(\bfI + \lambda^{-2} \bfH_\bfQ)] \\
 \leq & \> n + \lambda^{-2} \trace(\bfH_\bfQ - \widehat\bfH_\bfQ).
\end{aligned}\]
Therefore, $\mc{E}_1 \leq \lambda^{-2}\theta_k$.
{For the third term, notice that $$\post - \posth = \lambda^{-2}\bfQ^{1/2} \left( \left( \bfI + \lambda^{-2}\bfH_\bfQ\right)^{-1} - \left( \bfI + \lambda^{-2}\widehat \bfH_\bfQ\right)^{-1}\right) \bfQ^{1/2}$$ and let $\bfD =(\bfI+ \lambda^{-2}\bfH_\bfQ)^{-1} - (\bfI + \lambda^{-2} \widehat\bfH_\bfQ)^{-1}$. Then}
\[\| \spost -{\bfs_k}\|_{\post^{-1}}^2 = \widehat\bfb\t\bfD(\bfI + \lambda^{-2} \bfH_\bfQ)\bfD \widehat\bfb \leq \|\bfD\widehat\bfb\|_2 \| (\bfI + \lambda^{-2} \bfH_\bfQ)\bfD \widehat\bfb\|_2, \]
where {$\widehat\bfb = \bfQ^{1/2}\bfA\t\bfR^{-1}\bfb$}. The inequality is due to Cauchy-Schwartz. Using \cref{eqn:invpert}, we can bound
\[ \|\bfD\widehat\bfb\|_2  \leq \frac{\omega_k\|{\widehat \bfb}\|_2}{ \lambda^{2}+\omega_k}.\]
Next, with $\bfE = \lambda^{-2}(\bfH_\bfQ - \widehat\bfH_\bfQ)$, consider {the simplification
\[  (\bfI + \lambda^{-2} \bfH_\bfQ)\bfD = -\bfE(\bfI + \lambda^{-2} \widehat\bfH_\bfQ)^{-1},\]
so that $\| (\bfI + \lambda^{-2} \bfH_\bfQ)\bfD \widehat\bfb\|_2 \leq \lambda^{-2} \omega_k \|\widehat\bfb\|_2 $.}
Here, we have used submultiplicativity and  the fact that singular values of $(\bfI + \lambda^{-2}\widehat\bfH_{\bfQ})^{-1}$ are at most $1$. We also see that $\|\widehat\bfb\|_2 = \alpha_1\beta_1$. Putting everything together, we see
\[  \mc{E}_3 \leq \frac{\lambda^{-2}\omega_k^2\alpha_1^2\beta_1^2}{\lambda^{2} + \omega_k} .\]
Gathering the bounds for $\mc{E}_1$, $\mc{E}_2$ and $\mc{E}_3$ we have the desired result.
\end{proof}

\begin{proof}[\cref{thm:dkl}]
The error in the KL-divergence satisfies
\[ |D_\text{KL} - \widehat{D}_\text{KL}| \leq \mc{E}_1 + \mc{E}_2 + \mc{E}_3,\]
where
\begin{eqnarray*}
\mc{E}_1  = &\> \frac12 |\trace(\bfI + \lambda^{-2}\bfH_{\bfQ})^{-1} - \trace (\bfI + \lambda^{-2}\widehat\bfH_{\bfQ})^{-1}| \\
\mc{E}_2 = & \> \frac12 | \log\det(\bfI + \lambda^{-2}\bfH_{\bfQ}) - \log\det(\bfI+\lambda^{-2}\widehat \bfH_\bfQ)| \\
\mc{E}_3 = & \> \frac12 \lambda^2 | (\spost-\bfmu)\t\bfQ^{-1}(\spost-\bfmu) - (\bfs_k-\bfmu)\t\bfQ^{-1}(\bfs_k-\bfmu)|.
\end{eqnarray*}
We tackle the first two terms together. As in the proof of \cref{thm:dklacc}, let $\bfM = \lambda^{-2}(\bfH_{\bfQ}) $, then with $\bfP = \widehat\bfV_k\widehat\bfV_k\t$ we have $\lambda^{-2}\widehat\bfH_\bfQ = \bfP\bfM\bfP$. Apply the first and the third parts of \cref{lemma:trdet} to obtain
\[ \mc{E}_1 \leq \frac{ \lambda^{-2}}{2}\trace(\bfH_{\bfQ} - \widehat\bfH_{\bfQ})  \qquad  \mc{E}_2 \leq\frac{  \lambda^{-2}}{2} \trace(\bfH_{\bfQ} - \widehat\bfH_{\bfQ}).\]
For the third term, let $\spost  =  \bfs_k + \bfe$, then
\[\mc{E}_3 =  \frac12 \lambda^2 | (\spost-\bfs_k)\t\bfQ^{-1}(\spost-\bfmu) + (\bfs_k-\bfmu)\t\bfQ^{-1}\bfe| . \]
Notice that $\bfe =  \bfs_{\rm post} - \bfs_k  = (\post - \posth) \bfA\t\bfR^{-1}\bfb$.  Let
$$\widehat\bfb \equiv \bfQ^{1/2}\bfA\t\bfR^{-1}\bfb = \alpha_1\beta_1 \bfQ^{1/2}\bfv_1,$$  and write
$$\bfQ^{-1/2}\bfe =   \left( (\lambda^2\bfI + \bfH_{\bfQ})^{-1} - (\lambda^2\bfI+ \widehat\bfH_{\bfQ})^{-1} \right) \widehat{\bfb}.$$
So, the submultiplicative inequality and \cref{eqn:invpert} implies
\[\begin{aligned} \|\bfQ^{-1/2}\bfe\|_2 \leq & \>   \lambda^{-2} \| (\bfI + \lambda^{-2}\widehat\bfH_{\bfQ})^{-1} - (\bfI+ \lambda^{-2}\bfH_{\bfQ})^{-1} \|_2\|\widehat{\bfb}\|_2 \\
 \leq & \> \lambda^{-2}\frac{\omega_k}{\lambda^2 + \omega_k} \alpha_1\beta_1\,,
\end{aligned}\]
where we have used \cref{appendix:bound}.
Next, applying the Cauchy-Schwartz inequality
\[|\bfe\t \bfQ^{-1} (\spost - \bfmu)| \leq\| \bfQ^{-1/2}\bfe\|_2 \|\bfQ^{-1/2}(\spost - \bfmu)\|_2.  \]
Then, rewriting $\spost = \bfmu + \post\bfA\t\bfR^{-1}\bfb$, we have
\[ \|\bfQ^{-1/2}(\spost - \bfmu)\|_2 = \|(\bfI + \lambda^{-2}\bfH_{\bfQ})^{-1} \widehat\bfb\|_2 \leq \|\widehat\bfb\|_2 = \alpha_1\beta_1,\]
since the singular values of $(\bfI + \lambda^{-2}\widehat\bfH_{\bfQ})^{-1}$ are less than $1$. The other term is bounded in the same way. So, we have
\[ \mc{E}_3  \leq \frac{\omega_k}{\lambda^2 + \omega_k} \alpha_1^2\beta_1^2. \]
Putting everything together along with
 $\mc{E}_1 + \mc{E}_2 \leq \lambda^{-2} \theta_k$
gives the desired result.
\end{proof}

\subsection{Proofs of \cref{sec:sampling}}
\label{sec:appendix3}

\begin{proof}[\cref{thm:sample}]
By {the} triangle inequality{, we have}
\[  \|\bfs - \widehat\bfs\|_{\lambda^2\bfQ^{-1}} \leq \|\spost-{\bfs_k}\|_{\lambda^2\bfQ^{-1}} + \|\bfS\bfepsilon - \widehat\bfS\bfepsilon\|_{\lambda^2\bfQ^{-1}}.\]
Similar to previous proofs, we use $\spost-\bfs_k = (\post - \posth) \bfA\t\bfR^{-1}\bfb = \lambda^{-2}\bfQ^{1/2}\bfD \widehat\bfb,$ where $\bfD = (\bfI + \lambda^{-2}\bfH_{\bfQ})^{-1} - (\bfI+ \lambda^{-2}\widehat\bfH_{\bfQ})^{-1}$ to get
\[ \|\spost-\bfs_k\|_{\lambda^2 \bfQ^{-1}}^2  = \widehat\bfb\t \bfD \bfQ^{1/2} \lambda^{-2} (\lambda^2 \bfQ^{-1}) \lambda^{-2}  \bfQ^{1/2}\bfD\widehat\bfb = \lambda^{-2} \|\bfD\widehat\bfb\|_2^2. \]
Thus,
{\[ \|\spost-{\bfs_k}\|_{\lambda^2\bfQ^{-1}} = \lambda^{-1}\|\bfD\widehat\bfb\|_2 \leq \lambda^{-1}\frac{\omega_k  \alpha_1\beta_1}{\lambda^2 + \omega_k}, \]}
For the second term, we write
$$\lambda\bfQ^{-1/2}(\bfS -\widehat\bfS)\bfepsilon =  \left[(\bfI + \lambda^{-2}\bfH_{\bfQ})^{-1/2} - (\bfI+ \lambda^{-2}\widehat\bfH_{\bfQ})^{-1/2} \right]\bfepsilon.$$
Then, applying submultiplicativity
\[ \|\bfS\bfepsilon -\widehat\bfS\bfepsilon\|_{{\lambda^2}\bfQ^{-1}} { =  \| \lambda\bfQ^{-1/2}(\bfS -\widehat\bfS)\bfepsilon\|_2}\leq \|(\bfI + \lambda^{-2}\widehat\bfH_{\bfQ})^{-1/2} - (\bfI+ \lambda^{-2}\bfH_{\bfQ})^{-1/2}\|_2{\|\bfepsilon\|_2}. \]
When we apply~\cite[Theorem X.1.1 and (X.2)]{bhatia2013matrix}, we have
\[ \|(\bfI + \lambda^{-2}\widehat\bfH_{\bfQ})^{-1/2} - (\bfI+ \lambda^{-2}\bfH_{\bfQ})^{-1/2}\|_2\leq \|\bfD\|_2^{1/2}.\]
From \cref{eqn:invpert}, $\|\bfD\|_2 \leq\omega_k/(\lambda^2 + \omega_k)$. Plugging this in gives the desired result.
\end{proof}

\bibliography{7references}
\bibliographystyle{abbrv}

\end{document}